\documentclass[12pt]{article}

\usepackage{color}

\usepackage{mathrsfs}
\usepackage{lscape}
 \usepackage[ansinew]{inputenc}
 \usepackage{multicol}
 \usepackage[dvips]{graphicx}
\usepackage[psamsfonts]{amssymb}
 \usepackage[all]{xy}

\usepackage[psamsfonts]{eucal}
\usepackage{amssymb, amsmath}
\usepackage{amsthm}
\usepackage{geometry}
\usepackage{enumerate}
\usepackage{float}
\usepackage{latexsym,amsmath,amsthm,verbatim,ifthen,amssymb,graphicx,color}
\usepackage[all]{xy}
\usepackage{color}
\theoremstyle{plain}

\newtheorem{proposition}{Proposition}[section]
\newtheorem{theoremb}[proposition]{Theorem}
\newtheorem{lemma}[proposition]{Lemma}
\newtheorem{corollary}[proposition]{Corollary}

\theoremstyle{definition}
\newtheorem{definition}[proposition]{Definition}
\newtheorem{example}[proposition]{Example}

\theoremstyle{remark}

\newcommand{\secref}[1]{Section~\ref{#1}}

\newcommand{\thmref}[1]{Theorem~\ref{#1}}
\newcommand{\propref}[1]{Proposition~\ref{#1}}

\newcommand{\corref}[1]{Corollary~\ref{#1}}

\newcommand{\defref}[1]{Definition~\ref{#1}}

\def\cA{{\mathcal A}}

\def\cC{{\mathcal C}}

\def\cE{{\mathcal E}}

\def\cG{{\mathcal G}}

\def\cI{{\mathcal I}}

\def\cL{{\mathcal L}}

\newcommand{\bz}{\mathbb Z}
\newcommand{\bq}{\mathbb Q}
\def\L{\mathbb{L}}

\def\Q{\mathbb{Q}}

\def\im{{\rm Im\,}}
\def\hL{{\widehat{\mathbb L}}}

\def\ad{{\rm ad}}
\def\id{{\rm id}}

\def\ov{\overline}

    \newcommand{\lasu}{{\mathfrak{L}}}
    \newcommand{\fracd}{{\mathfrak{D}}}

    \newcommand{\sdelta}{\Delta}
    \newcommand{\sfdelta}{\dot\Delta}

 \newcommand{\lib }{\mathbb{L}}

\newcommand{\catcdga}{\operatorname{{\bf CDGA}}}

\newcommand{\CDGC}{\operatorname{{\bf CDGC}}}

\newcommand{\catdglco}{\operatorname{{\bf DGLC}}}
\newcommand{\catdgl}{\operatorname{{\bf DGL}}}

\newcommand{\catss}{\operatorname{{\bf SimpSet}}}

 \newcommand{\MC}{\operatorname{{\rm MC}}}
\newcommand{\mc}{{\MC}}

   \newcommand{\libc}{{\widehat\lib}}
\def\Lc{\mathbb{L}^{\!c}}

\title{Lie models of simplicial sets and representability of the Quillen functor}
\author{Urtzi Buijs, Yves F\'elix, Aniceto Murillo and Daniel Tanr\'e\footnote{The  authors have been partially supported by the MINECO grant MTM2013-41768-P. The first author has also been partially supported by a U-Mobility Programme grant of the
European Union's Seventh Framework Programme, G.A. no. 246550 and by the Ram\'on y Cajal MINECO program.  The first and third authors have also been partially supported by the Junta de Andaluc{\'\i}a grant FQM-213. \vskip 1pt {\em 2010 Mathematics Subject
Classification}. Primary: 55P62; Secondary: 55P30.\vskip
 1pt
 {\em Key words and phrases}: Rational homotopy theory. Lie models of Simplicial sets. Quillen functor.}}

\begin{document}

\maketitle

\begin{abstract}
Extending the model of the interval, we explicitly define for each $n\ge 0$ a free complete differential graded Lie algebra $\lasu_n$ generated by the simplices of $\Delta^n$, with desuspended degrees,   in which the vertices are Maurer-Cartan elements and the differential extends the simplicial chain complex of the standard n-simplex. The family $\{\lasu_\bullet\}_{n\ge 0}$  is endowed with a cosimplicial differential graded Lie algebra structure which we use to construct two adjoint functors
$$\xymatrix{ \catss& \catdgl \ar@<1ex>[l]^(.40){\langle\,\cdot\,\rangle}
\ar@<1ex>[l];[]^(.55){\lasu}\\}
$$given by  $\langle L\rangle_\bullet=\catdgl (\lasu_\bullet,L)$ and $
\lasu(K)=\varinjlim_K\lasu_{\bullet}
$. This new tools let us extend Quillen  rational homotopy theory approach to any simplicial set $K$ whose path components are non necessarily simply connected.

We prove that $\lasu (K)$ contains a model of each component of $K$. When $K$ is a 1-connected  finite simplicial complex, the Quillen model of $K$ can be extracted from $\lasu (K)$.   When $K$ is connected then, for a perturbed differential $\partial_a$, $H_0(\lasu (K),\partial_a)$ is the Malcev Lie completion of $\pi_1(K)$. Analogous results are obtained for the realization $\langle L\rangle$ of any complete DGL.
\end{abstract}

\section*{Introduction}
In \cite{LS}, R. Lawrence and D. Sullivan raise the following observation and subsequent general questions: the rational singular chains on a cellular complex are naturally endowed with a structure of cocommutative, coassociative infinity coalgebra and hence, taking the commutators of a ``generalized bar construction'' it should give rise to a complete  differential graded Lie algebra (DGL henceforth). What is the topological and geometrical meaning of this DGL? Are there closed formulae for its differential? Allowing 1-cells, what is the relation of this DGL with the fundamental group of the given complex? In the same reference they carefully construct such a DGL for the interval.

 Here we attack these problems for any finite simplical complex $K$. In fact, we construct  in a functorial way
  a complete differential graded Lie algebra   $\lasu(K)=(\widehat{\mathbb L}(V_K), \partial)$ which is the only (up to isomorphism) free complete differential graded Lie algebra  for which
  $V_K$ together with the linear part $\partial_1$ of  $\partial$ is the desuspension of the simplicial chain complex of $K$, and any generator of degree $-1$, (i.e., any $0$-simplex) is a Maurer-Cartan element. The departure point is the construction of a family $\lasu_n:=\lasu(\Delta^n)$ of such DGL's   for the standard simplices, together with coface and codegeneracy operators, for which we prove:

\vspace{3mm}\noindent {\bf Theorem A.} \label{thmintro:simplicialdgl}
{\sl  The family $\lasu_{\bullet}:=\{\lasu({\Delta^n}),\partial\}_{n\ge 0}$ form a cosimplicial differential graded Lie algebra.
}
\vspace{3mm}

We also prove uniqueness of such a family as Theorem \ref{unicidadcompleta} and give as examples, explicit models $\lasu_2$ and $\lasu_3$ for the triangle and the tetrahedron in Proposition \ref{prop:eltriangulo} and Example \ref{ejemplo} respectively.

On the one hand, this cosimplicial structure let us finally define,
$$
\lasu(K)=\varinjlim_K\lasu_{\bullet},
$$
for any simplicial set $K$.  Whenever $K$ is in particular a finite simplicial complex, $\lasu(K)$ is trivially isomorphic to a sub DGL of a certain $\lasu_n$ as $K\subset \Delta^n$ for some $n$.

On the other hand, the cosimplicial DGL $\lasu_{\bullet}$   defines a functor $\langle \,\cdot\,\rangle$  from the category {\bf DGL}  of complete differential graded Lie algebras to the category  {\bf SimpSet} of simplicial sets via
$$\langle \, L\,\rangle = \mbox{\bf DGL}(\lasu_\bullet, L).$$

Then, we prove:

 \vspace{3mm}\noindent {\bf Theorem B.} \label{adjoint}
{\sl The following are adjoint functors $$\xymatrix{ \catss& \catdgl. \ar@<1ex>[l]^(.40){\langle\,\cdot\,\rangle}
\ar@<1ex>[l];[]^(.55){\lasu}\\}
$$
}

These functors not only shed light on our original questions but also let us develop an  Eckmann-Hilton dual of the Sullivan geometrical approach to rational homotopy theory, even for  non simply connected complexes. Here, $\lasu_\bullet$ plays the dual role of the simplicial differential graded algebra  $\mathscr{A}_\bullet=A_{PL}(\Delta^\bullet)$ of PL-forms on the standard simplices  \cite{FHT,Su,tan}. Recall that the Quillen approach \cite{qui} is based on the construction of several equivalences between homotopy categories joining the homotopy category of rational $2$-reduced
simplicial sets and that of $1$-reduced differential graded Lie algebras.

Our work, although using a completely different method, it extends Quillen theory in several ways.

First of all, if $X$ is a simply connected simplicial set, we denote by $Y$  the associated 2-reduced simplicial set. In the text, we abuse notation by  denoting  $\lambda(X)$ the image of $Y$ by the Quillen construction. We prove:

\vspace{3mm}\noindent {\bf Theorem C.}  \label{thmintro:quillen}
{\sl Let $K$ be a 1-connected finite simplicial complex. Then, for any vertex $a\in K$, there is a quasi-isomorphism of
differential graded Lie algebras,
 $$(\lasu(K),\partial_a)\stackrel{\simeq}{\longrightarrow}\lambda (K).$$
 In particular,   $H_n(\lasu (K), \partial_a) \cong \pi_{n+1}(K)\otimes \mathbb Q$, $n\ge 1$.}

\vspace{3mm}

Here $\partial_a=\partial+\ad_a$ is the twisting  of the original differential $\partial$ of $\lasu(K)$  via the Maurer-Cartan element $a$.

\vspace{3mm} On the other hand, concerning the realization functor, we show that, for a non negatively graded  DGL $L=L_{\ge 0}$, there are explicit isomorphisms,
 $$
 \pi_n\langle L\rangle\cong H_{n-1}(L),\quad n\ge 1,
 $$
 in which the group structure on $H_0(L)$ is given by the Baker-Campbell-Hausdorff product (see Proposition \ref{componentes}).

 Moreover, recall the dual Sullivan geometric realization \cite{FHT,Su}  of a commutative differential graded algebra  $A$ given by
$$
\langle A\rangle_S = \mbox{\bf CDGA} (A, \mathscr A_\bullet).$$
Then:

\vspace{3mm}\noindent {\bf Theorem D.} \label{thmintro:realizations} {\sl If $L$ is a non negatively graded finite type DGL, then there is a homotopy equivalence
$\langle L\rangle \simeq \langle {\mathcal C}^*(L)\rangle_S$.}
\vspace{3mm}

Here, ${\mathcal C}^*$ denotes the cochain functor. In particular our realization also equals Quillen realization for finite type $1$-reduced DGL's in view of the main result in   \cite{ma}.
Moreover, in \cite{getz} Getzler has constructed another functor $\langle\,\cdot\,\rangle_G$ from {\bf DGL} to {\bf SimpSet},
$$\langle L\rangle_G = \mc(L\widehat{\otimes} \mathscr A_\bullet)\,,$$
in which $\mc$ denotes the set of Maurer-Cartan elements. By \cite[Proposition 1.1]{getz}, $\langle L\rangle_G = \langle {\mathcal C}^*(L)\rangle_S$ when $L$ is a non negatively graded finite type DGL, and thus, for these DGL's, all realizations coincide (see also \cite[Cor. 1.3]{Ber}).

Finally, we answer the last question which we begin with and explicitly describe the relation between the fundamental group of a complex and its associated DGL:

\vspace{3mm}\noindent {\bf Theorem E.} \label{thmintro:Malcev}
{\sl  If $K$ is a connected  finite simplicial complex, then
$H_{0}(\lasu(K), \partial_a)$
is the Malcev Lie completion of the fundamental group of $K$.}

\vspace{3mm}

This result implies that, unlike the classical Quillen approach and their known interesting extensions \cite{BM,lamar}, our model functor reflects geometrical properties of non-nilpotent spaces.

We also point out that our constructions apply to non necessarily connected complexes. Indeed, for any finite complex $K$, the connected components of $K_+$, the disjoint union of $K$ and a point, correspond to the classes of Maurer-Cartan elements of $\lasu(K)$ modulo the gauge action \cite{ulmc}.
Moreover, if $a$ is a non-zero  Maurer-Cartan element, then $(\lasu(K),\partial_a)$,
is a ``model'' of the corresponding path component of $K$.

On the other hand, in Theorem \ref{localizar} we prove that, for a general DGL $L$, the components of the realization $\langle L\rangle$ correspond also to  the classes of Maurer-Cartan elements of $L$ modulo the gauge action. Moreover, if $a$ is a  Maurer-Cartan element, the corresponding component of $\langle L\rangle$ is the realization $\langle L^{(a)}\rangle$ of the ``localization of $L$ at $a$''. In other words, $\langle L\rangle\simeq \dot\cup_{a\in\widetilde\MC(L)} \langle L^{(a)}\rangle$. This extends \cite[Theorem 5.5]{BM2} to any (complete) DGL.

We remark that the definition and main properties of the model and realization functor purposely stay in the Lie side and do not rely upon any connection to commutative algebras via cochain functors. The main reason is the following: in forthcoming work \cite{ul}, we endow the category of complete DGL's with a model structure which makes the model and realization functors a Quillen pair. This  structure geometrically shapes complete Lie algebras as it is the result of ``transfering'' the classical model structure on simplicial sets. With this, weak equivalences and fibrations  properly contain the weak equivalences and fibrations of the model structure defined in \cite{lamar} via the cochain functor.

\vspace{3mm} This paper is organized as follows:

\secref{sec:prems} is devoted to recall the Baker-Campbell-Hausdorff product and its connection with the basic properties of Lawrence-Sullivan model $\lasu_1$ for the interval \cite{LS}.
 In   \secref{sec:delta} we  extend this construction and build the
 sequence $\{\lasu_n\}_{n\ge 0}$ of  \emph{compatible models} for  $\Delta^n$. We establish some technical properties and also prove uniqueness.

In  Section 3 we prove the existence of compatible symmetric models, i.e., models $(\lasu_n, \partial)$
that are $\Sigma_{n+1}$-DGL with an action that permutes the inclusions of the different copies of $\lasu_{n-1}$. This leads to the existence of the cosimplicial DGL structure on $\lasu_{\bullet}$. All of this proves Theorem A.

In Sections 4 and 5 we   define  the {\em model} $\lasu$ and {\em realization} $\langle \,\cdot\,\rangle$ functors, we prove Theorem B, and
give some properties and examples.

 Section 6 is devoted to introduce the material on  differential graded Lie coalgebras and  the transfer theorem that we use   in Section 7 to build a sequence of compatible models from a transfer process applied to a classical diagram of the form
$$
\xymatrix{
{}
&
 \ar@(ul,dl)@<-5.5ex>[]_{\phi}
 &
  A_{PL}(K) \ar@<0.75ex>[r]^-{p}
  &
   {C^*(K).} \ar@<0.75ex>[l]^-{i} }
$$
Using the uniqueness Theorem \ref{unicidadcompleta}, we establish relations   between the classical rational models of Sullivan and Quillen and the models constructed in the previous sections. In particular
  we prove
Theorem C as \thmref{thm:quillenetlasu}.

Section 8 contains  the correspondence between our realization and Sullivan's realization.  Theorem D corresponds to \thmref{lestrella}. Finally, Theorem E is proved in Section 9.

 The authors would like to thank Benoit Fresse for helpful conversations.
%%%%%%%%%%%%%%
{\small
\tableofcontents}

%%%%%%%%%%%%%%%%%

\section{The BCH product and the Lawrence-Sullivan model for the interval}\label{sec:prems}

Throughout this paper we assume that $\bq$ is the base field. A \emph{graded Lie algebra} consists of a $\mathbb Z$-graded vector space $L=\oplus_{p\in\bz}L_p$ together with a bilinear product called \emph{the Lie bracket} denoted by $[-,-]$ such that $[x,y] = -(-1)^{\vert x\vert \vert y\vert} [y,x]$ and
$$(-1)^{\vert x\vert \vert z\vert} \bigl[x,[y,z]\bigr] + (-1)^{\vert y\vert \vert x\vert} \bigl[y, [z,x]\bigr] + (-1)^{\vert z\vert  \vert y\vert} \bigl[z,[x,y]\bigr]  = 0.$$
Here $\vert x\vert$ denotes the degree of $x$.

A \emph{differential graded Lie algebra}  is a graded Lie algebra $L$ endowed with a linear derivation $\partial$  of degree $-1$ such that $\partial^2= 0$. It is called \emph{free} if $L$ is free as a Lie algebra, $L = \mathbb L(V)$ for some graded vector space $V$.

The \emph{completion} $\widehat{L}$ of a graded Lie algebra $L$ is
$$\widehat{L} = \varprojlim_n L/L^n$$
where $L^1 = L$, $L^n = [L, L^{n-1}]$  for $n\geq 2$, and the limit is taken on the topology arising from this filtration. A Lie algebra $L$ is called \emph{complete} if $L$ is isomorphic to its completion. From now on, and unless explicitly stated otherwise, by a DGL we mean a complete differential graded Lie algebra.

A Maurer-Cartan element is an element $a\in L_{-1}$ such that $\partial a+ \frac{1}{2}[a,a] = 0$. Denote by ${\MC}(L)$ the set of Maurer-Cartan elements.
 These are preserved by DGL morphisms. In particular, if $L = (\widehat{\mathbb L}(V),\partial)$ is a complete free DGL, and $\theta$ is a derivation satisfying $\theta (V) \subset \widehat{\mathbb L}^{\geq 2}(V)$ and $[\theta,\partial]= 0$, then $e^{\theta}=\sum_{n\ge 0}\frac{\theta^n}{n!}$ is an automorphism of $L$ and so, if $a \in {\MC}(L)$, then $e^{\theta}(a)$ is also a Maurer-Cartan element.

Given $(\libc(V),\partial)$ a complete free DGL and $v\in V$, we will often write $\partial v=\sum_{n\ge1}\partial_nv$ where $\partial_nv\in\mathbb{L}^{n}(V)$.

Let $(L,\partial)$ be a DGL and $a\in {\MC}(L)$. Then,  the derivation, $\partial_a=\partial+\ad_a$ is again a differential on $L$.

Given $L$ a complete DGL, the {\em gauge action} $\cG$ of $L_0$ on ${\MC}(L)$ determines an equivalence relation among Maurer-Cartan elements defined as follows (see for instance \cite[\S4]{mane}): given $x\in L_0$ and $a\in {\MC}(L)$,
$$
x\,\cG\, a=e^{\ad_x}(a)-\frac{e^{\ad_x}-1}{\ad_x}(\partial x).
$$
Here and from now on, $1$ inside an operator will denote the identity. Explicitly,
$$
x\,\cG\, a=\sum_{i\ge0}\frac{\ad_x^i(a)}{i!}-\sum_{i\ge0}\frac{\ad_x^{i}(\partial x)}{(i+1)!}.
$$
We denote the quotient set by $\widetilde\mc(L)=\mc(L)/\mathcal G$.
Geometrically \cite{kont,LS}, interpreting Maurer-Cartan elements as points in a space, one thinks of $x$ as a flow taking   $x\,\cG\, a$ to $a$ in unit time. In topological terms \cite{BM2}, the points $a$ and $x\,\cG\, a$ are in the same path component.

\medskip

 Let $L$ be a complete Lie algebra concentrated in degree $0$.
 We denote by $UL$ its enveloping algebra, by $I_L$ its augmentation ideal and by $\widehat{UL}$ and $\widehat{I_L}$
 the completions of $UL$ and $I_L$ with respect to the powers of $I_L$,
 $$\widehat{UL} = \varprojlim_n UL/I_L^n , \hspace{5mm} \widehat{I_L} = \varprojlim_n I_L/I_L^n.$$
 Denote finally by $G_L = \{ x\in \widehat{UL} \, \vert\, \Delta (x) = x\widehat \otimes x\}$
 the group of grouplike elements in $\widehat{UL}$.
 Moreover, the injection of $L$ in the set of primitive elements in $\widehat{UL}$ is an isomorphism
 and the functions exp and log give inverse bijections between $L$ and $G_L$.
 This induces a product on $L$, called the \emph{Baker-Campbell-Hausdorff product,}  BCH product henceforth, defined by
 $$a*b = \log (\exp (a) \cdot \exp (b)).$$
 Note that $a*(-a)=0$. Therefore, $-a$ is the inverse of $a$ for the BCH product and we also use the notation
 $-a=a^{-1}$.

 As the law in $G_{L}$ is associative, the BCH product is also associative.
 An explicit form of the product is given by the Baker-Campbell-Hausdorff formula
 $$a*b = a+b + \frac{1}{2} [a,b] + \frac{1}{12} \bigl[a,[a,b]\bigr] - \frac{1}{12} \bigl[b,[a,b]\bigr] + \cdots $$
 It follows from the Jacobi identity that in the Lie algebra of derivations of $L$ we have
 $\ad_{a*b} = \ad_{a} * \ad_{b}$. Hence $e^{\ad_{a*b}}=e^{\ad_{a}}\circ e^{\ad_{b}}$.

 Note that the BCH product is compatible with the gauge action on $\mc(L)$; i.e.,
 if $y\in L_{0}$ and $a \in\mc(L)$, we have
 $$(x*y)\cG a=x\cG(y\cG a).$$
We  also need the following property.

 \begin{proposition}\label{prop:expadx}
  Let $L$ be a complete DGL and let $x,y\in L_0$. Then,
$$
x*y*(-x) = e^{\ad_x}(y).
$$
 \end{proposition}
With the previous convention, the formula also reads
$$x*y*x^{-1}= e^{\ad_x}(y).$$

 \begin{proof}
 First, we note that, in  $\widehat{UL}$,
$$
 e^{\ad_x}(y) = e^xye^{-x}.
$$
Indeed,
 \begin{eqnarray*}
 e^{\ad_x}(y)
 &=&
 \sum_{n=0}^\infty \frac{\ad_x^n(y)}{n!} = \sum_{n=0}^\infty \frac{1}{n!} \sum_{i=0}^n (-1)^i
  \binom{n}{i} x^{n-i}yx^i\\
 &=&
  \sum_{i=0}^\infty\sum_{n=i}^\infty \frac{x^{n-i}}{(n-i)!}\, y \,\frac{(-x)^i}{i!} = e^xye^{-x}.
 \end{eqnarray*}
 Replacing $y$ by $e^y$, we deduce
 \begin{equation}\label{equa:expadx2}
 e^{\ad_x}(e^y)=e^xe^ye^{-x}.
  \end{equation}

  In a second step, we prove the equality
  \begin{equation}\label{equa:expadxn}
  (e^{\ad_x})(y^n) = \left( e^{\ad_x}(y)
 \right)^n.
     \end{equation}
 On the left hand side, the term of length $k$, $k\geq 0$, in $x$ equals
   $$
   \frac{\ad_x^k(y^n)}{k!}.
   $$
In the right hand side, this term is
 $$
  \sum_{k_1+ \cdots +k_n = k} \frac{\ad_x^{k_1}(y)}{k_1!} \cdots \frac{\ad_x^{k_n}(y)}{k_n!}.
$$
  As $\ad_{x}$ is a derivation, both terms coincide and  the equality (\ref{equa:expadxn}) is proved.
Therefore,
\begin{equation}\label{equa:expexpadx}
e^{\ad_x}(e^y) =   \sum_{n\geq 0} \frac{(e^{\ad_x})(y^n)}{n!} =  \sum_{n\geq 0} \frac{\left( e^{\ad_x}(y)\right)^n}{n!} =
 e^{e^{\ad_x}(y)}.
\end{equation}

Finally, the proposition follows from
$$
x*y*(-x)
=
 \log (e^xe^ye^{-x})
 =_{(\ref{equa:expadx2})}
 \log\left(e^{\ad_{x}}(e^y)\right)
 =_{(\ref{equa:expexpadx})}
 \log (e^{e^{\ad_x}(y)})
 =
  e^{\ad_x}(y).
  $$
 \end{proof}

The   construction of a model for the interval was first introduced in \cite{LS}.

 \begin{definition}\label{def:LSconstruction}
\emph{The Lawrence-Sullivan model for the interval}, LS-interval henceforth, is the complete DGL
$$(\lasu_{1}, \partial)=(\widehat{\mathbb L} (a,b,x),\partial),$$
in which $a$ and $b$ are Maurer-Cartan elements, $x$ is of degree~0 and
\begin{equation}\label{equa:dxLSinterval}
\partial x = \ad_xb + \sum_{n=0}^\infty \frac{B_n}{n!}\,  \ad_x^n (b-a) = \ad_xb + \frac{\ad_x}{e^{\ad_x}-1} (b-a),
\end{equation}
 where the $B_n$'s are the Bernoulli numbers.
\end{definition}
Using the identity
$$
\left( \frac{-x}{e^{-x}-1}\right) = x+ \left( \frac{x}{e^x-1}\right),
$$
we may also write
\begin{equation}\label{dxa}
\partial x = ad_xa + \frac{ad_{-x}}{e^{-ad_x}-1}(b-a).
\end{equation}

Let $(\libc(a_0,a_1, a_2,x_1,x_2),\partial)$ be two glued LS-models of the interval. That is, $a_0, a_1$ and $a_2$ are Maurer-Cartan elements, $\partial x_1 = \ad_{x_1}(a_1) + \frac{ \ad_{x_1}}{e^{\ad_{x_1}}-1} (a_1-a_0)$ and
  $\partial x_2 = \ad_{x_2}(a_2) + \frac{ \ad_{x_2}}{e^{\ad_{x_2}}-1} (a_2-a_1)$.
These two models give a model for the subdivision of an interval, as follows.

 \begin{theoremb} \label{prop:subdivision} {\rm \cite[Theorem 2]{LS}}
 The map
 $$\gamma\colon (\hL(a,b,x),\partial) \longrightarrow(\hL(a_0,a_1,a_2, x_1, x_2),\partial),$$
defined by $\gamma(a)= a_0$, $\gamma(b)= a_2$ and $\gamma(x) = x_1*x_2$, is a DGL morphism.
\end{theoremb}

Finally, we  will make use of the following results (for the first cf. also \cite{pepdt}).

\begin{theoremb}\label{unicols} Let $(\libc(a,b,x),\partial')$ be a complete Lie algebra
in which $a,b$ are Maurer-Cartan elements and the linear part of
the differential satisfies $\partial'_1x=b-a$. Then  $(\libc(a,b,x),\partial')=(\lasu_{1}, \partial)$.
\end{theoremb}

\begin{proof} Let $\partial= \sum_{n\geq 1}\partial_n$,  with  $\im \partial_n \subset \lib^{n}(a,b,x)$, be the differential in $\lasu_{1}$. We show that $\partial_nx=\partial'_nx$ for any $n\ge1$. For $n=1$ this is trivially true.
For  $n=2$, by using $\partial_{1}\partial'_{2}+\partial'_{2}\partial_{1}=0$, a short computation shows that
$$
\partial_2x=\partial_2'x=\frac{1}{2}[x,a+b].
$$
Assume $\partial_mx=\partial_m'x$, for $1\le m<n$. From $\partial^2=\partial'^2=0$ and the induction hypothesis we deduce that $\partial_nx-\partial'_nx$ is a decomposable $\partial_1$-cycle of degree $-1$.
Since $H_*(\libc(a,b,x),\partial_1)=\mathbb L(a)$ and $\libc(a,b,x)_0=\bq x$, we conclude that $\partial_nx-\partial'_nx=0.$
\end{proof}

\begin{proposition}\label{prop:dadb}
Let $L$ be a complete DGL which contains an LS-interval $(\widehat{\mathbb L} (a,b,x),\partial)$. Then, for any $v\in L$,
$$
\partial_a\,e^{\ad_x}(v)=e^{\ad_x}(\partial_bv).$$
In other words, the map
$$e^{\ad_x} {\colon} (L, \partial_b) \to (L, \partial_a)$$ is an isomorphism of DGL's.
\end{proposition}

\begin{proof}
In fact,
\begin{eqnarray*}
 \partial_{b}\,e^{-\ad_x} (v)
 &=&
  e^{-\ad_x}(\partial_{b}v) + (-1)^{\vert v\vert} e^{-\ad_x}\ad_v \frac{e^{\ad_x} - 1}{\ad_x} (\partial_{b}x)\\
  &=&
e^{-\ad_{x}}\left(\partial_{b}v+(-1)^{|v|}\ad_{v}(b-a)\right)\\
&=&
  e^{-\ad_{x}}(\partial_{a}v)
,
\end{eqnarray*}
where the first equality comes from \cite[Lemma 1]{LS} and the second is a direct computation using
(\ref{equa:dxLSinterval}).
\end{proof}

%%%%%%%%%%%%%%%%%%%%%%%%%%%%%%%%
\section{Sequences of compatible   models of $\mathbf\Delta$}\label{sec:delta}

For each $n\ge0$,  we consider the standard $n$-simplex $ \Delta^n$,
$$\Delta^n_p=\{(i_0,\dots,i_p)\mid 0\le i_0<\dots<i_p\leq n\},\quad\text{if} \quad p\leq n,$$
and $\Delta^n_p=\emptyset$ if $p>n$.
Let  $(\libc(s^{-1}\Delta^n),d)$ be the complete free DGL on the desuspended rational simplicial chain complex on  $\Delta^n$,
\begin{equation}\label{simpdif}
d a_{i_0\dots i_p} = \sum_{j=0}^p (-1)^j a_{i_0\dots \widehat{i}_j\dots i_p}.
\end{equation}
 Here, $a_{i_0\dots i_p}$ denotes the generator of degree $p-1$  represented by the $p$-simplex
 $(i_0,\dots,i_p)\in \Delta^n_{p}$.
    Henceforth, unless explicitly needed, we drop the desuspension sign to avoid unnecessary notation and write simply  $(\libc(\Delta^n),d)$.

For each $0\le i\le n$ consider the $i$-th coface affine map $
\delta_j\colon\Delta^{n-1}\to \Delta^{n},
$ defined on the vertices by,
$$
\delta_i(j)=\begin{cases}\,\,j,\,\,&\text{if $j<i$,}\\ j+1,\,\,&\text{if $j\ge i$}.
\end{cases}
$$
 We use the same notation for the induced DGL morphism,
\begin{equation}\label{coface}
\delta_i\colon (\libc(\Delta^{n-1}),d)\longrightarrow (\libc(\Delta^n),d),
\end{equation}
defined by
$$\delta_i(a_{j_0\dots j_p}) = a_{\ell_0\dots \ell_p}\quad\text{with}\quad \ell_k=\begin{cases} \,\, j_k,\,&\text{if $j_k<i$},\\ j_k+1,\,&\text{if $j_k\ge i$}.
\end{cases}
$$
Finally, we  denote by $\dot\Delta^n$   and  $\Lambda^n_i$ the boundary of $\Delta^n$ and the $i$-horn obtaining by removing the $i$-th coface from $\dot\Delta^n$.

\begin{definition}\label{modelo} A  {\em sequence of  compatible models of $\mathbf\Delta$} is  a family
$\{ (\lasu_{n}, \partial)=(\libc(\sdelta^{n}),\partial)\}_{n\ge 0}$
of DGL's satisfying the following properties:
 \begin{itemize}
 \item[(1)] For each $i=0,\dots,n$, the generator $a_i\in \sdelta^n_0$ is a Maurer-Cartan element, $\partial a_i=-\frac{1}{2}[a_i,a_i]$.
\item[(2)] The linear part $\partial_1$ of $\partial$ is precisely  $d$ as in (\ref{simpdif}).
\item[(3)] For each $i=0,\dots,n$, the coface maps, $\delta_i\colon (\libc(\sdelta^{n-1}),\partial)\to (\libc(\sdelta^{n}),\partial)$,
 are DGL morphisms.
    \end{itemize}
Each element $(\libc(\sdelta^{n}),\partial)$ of this sequence is called \emph{a model of $\Delta^n$},
which is thus implicitly endowed with models of $\Delta^q$, $q<n$, satisfying the compatibility condition (3) above.
\end{definition}

\begin{definition} A   sequence of models $\{(\libc(\sdelta^n),\partial)\}_{n\ge0}$   is called {\em inductive} if, for $n\ge 2$, we have
\begin{equation}
    \partial_{a_0} a_{0{\dots}n}\in\libc(\sfdelta^n).
\end{equation}
\end{definition}

\begin{theoremb}\label{existencia} There exists  sequences of compatible inductive models of $\mathbf\Delta$.
\end{theoremb}

Its proof uses the following key proposition: Let $(\widehat{\mathbb L}(V\oplus W), \partial)$ be a complete free DGL in which $V$ and $W$ are of finite dimension. We denote by $\cI$ the ideal generated by $W$, and by $\partial_1$ the linear part of the differential.

\begin{proposition}\label{auxi} With the above notations, if $\partial(\cI)\subset \cI$ and $H(W, \partial_1)= 0$, then
  the projection $(\widehat{\mathbb L}(V\oplus W), \partial) \to (\widehat{\mathbb L}(V), \partial)$ is a quasi-isomorphism.

 Moreover, every cycle in $\cI\cap \widehat{\mathbb L}^{\geq n}(V\oplus W)$ is a boundary $\alpha = \partial \beta$ with $\beta \in \cI\cap \widehat{\mathbb L}^{\geq n}(V\oplus W)$.
\end{proposition}

\begin{proof} Let $K$ be the kernel of the projection $(\mathbb L(V\oplus W), \partial_1) \to (\mathbb L(V), \partial_1)$. Since $H(W,\partial_1)= 0$, $K$ is acyclic. Now let $\alpha\in \cI$ be a  $\partial$-cycle. There exists $n$ such that
$\alpha=\alpha_{n}+\beta_{n+1}$,
with $\alpha_{n}\in \cI\cap \libc^n(W\oplus V)$ and $\beta_{n+1}\in \cI\cap \libc^{>n}(W\oplus V)$.
From $\partial \alpha=0$, we deduce $\partial_{1}\alpha_{n}=0$ and since $\alpha_n \in K\cap \mathbb L^n(W\oplus V)$  and $H(K, \partial_1)= 0$, there is an element $\gamma_{n}\in K\cap \mathbb L^n(W\oplus V) = \cI\cap \libc^n(W\oplus V)$ such that $\alpha_{n}=\partial_{1}\gamma_{n}$.
Thus, we have $\alpha-\partial \gamma_{n}\in \cI\cap \libc^{>n}(W\oplus V)$.
By iterating this process, we obtain a sequence $(\gamma_{j})_{j\geq n}$, such that
$\gamma_{j}\in \cI\cap \libc^j(W\oplus V)$ and $\alpha=\partial \left(\sum_{j\geq n}\gamma_{j}\right)$.

\end{proof}

\begin{corollary}\label{noseque}
If $(\libc(\Delta^n),\partial)$
is an inductive model, then:
\begin{enumerate}
\item[(i)] $H(\widehat{\mathbb L}(\Delta^n), \partial_{a_0}) = H(\widehat{\mathbb L}(\Lambda^i_n), \partial_{a_0}) = 0$ for any $i= 0,\dots , n$.
\item[(ii)] $
H(\libc(\dot\Delta^n),\partial_{a_0})=\lib(\Omega),\quad\text{with}\quad  \Omega=\partial_{a_0}a_{0\dots n}.
$
\end{enumerate}
\end{corollary}

\begin{proof} (i) This follows from Proposition \ref{auxi}. Let $W$ be  the vector space generated by the elements $a_i-a_0$ and the elements of higher degree. We verify directly that $H(W,d_1)= 0$ and that the ideal generated by $W$ is a differential ideal.

(ii) Write $(\libc(\Lambda^n_n),\partial_{a_0})=(\libc(\Q a_{0}\oplus W),\partial_{a_0})$, with
$W$ as above.

Let $(L,\partial)=(\libc(\Q a_{0}\oplus \Q u\oplus W),\partial)$ be the complete DGL,
with $|u|=n-2$, $\partial u=0$ and $\partial = \partial_{a_0}$ on $\mathbb Qa_0\oplus W$. This DGL satisfies the hypotheses of \propref{auxi} and by Lemma \ref{thel},
\begin{equation}\label{equa:lu}
H(L,\partial)=H(\libc(a_{0},u),\partial_{a_{0}})=\libc(u).
\end{equation}
We extend the canonical inclusion
$(\libc(\Lambda^n_n),\partial_{a_{0}})\to (\libc(\dot\Delta^n),\partial_{a_{0}})$
to a morphism,
$$f\colon (L,\partial)\to (\libc(\dot\Delta^n),\partial_{a_{0}}),$$
given by $f(u)=\partial_{a_{0}}a_{0\ldots n}$,
which is well defined thanks to the inductive hypothesis.
Observing its behavior on the generators we deduce that $f$ is a DGL isomorphism.
Using (\ref{equa:lu}), we get
$$\libc(u)\cong \libc(\Omega)=H(\libc(\dot\Delta^n),\partial_{a_{0}}).$$
\end{proof}

\begin{lemma} \label{thel} Let $(\widehat{\mathbb L}(\mathbb Q a \oplus V), \partial)$ be a DGL in which $a$ is a Maurer-Cartan element,  dim$\, V<\infty$ and $\partial (V)\subset \widehat{\mathbb L}(V)$. Then, the injection  $(\widehat{\mathbb L}(V), \partial)\hookrightarrow (\widehat{\mathbb L}(\mathbb Q a \oplus V), \partial)$ is a quasi-isomorphism.\end{lemma}

\begin{proof} Denote by $J$ the ideal generated by $V$. Since $V$ is finite dimensional,
 $$J = \widehat{\mathbb L}(\ad_a^qv, \vspace{2mm} \mbox{$v$ basis of $V$ and $q\geq 0$}).$$
 We then denote by $d_1$ the linear part of the differential $d$ in $J$.

We claim that for $v\in V$, and $q$ odd, $$d_1 \ad_a^q(v)= -\ad_a^{q+1}(v)+ (-1)^q \ad_a^q(\partial_1v).$$
This holds for $q= 1$ because $d_1 [a,v] = -\frac{1}{2}\bigl[[a,a],v\bigr] -[a, \partial_1 v ]= -\ad^2_a(v)-[a, \partial_1 v ]$.
Suppose this is true for $q-2$. Then, since $\ad_a^q(v) = \frac{1}{2}\bigl[[a,a], \ad_a^{q-2}(v)\bigr]$, we have
$$\begin{aligned}
d_1 \ad_a^q(v)&= \frac{1}{2} d_1 \bigl[[a,a], \ad_a^{q-2}(v)\bigr]= -\frac{1}{2}\bigl[[a,a], \ad_a^{q-1}(v)\bigr] +(-1)^q \ad_a^q(\partial_1v)\\&= -\ad_a^{q+1}(v)+(-1)^q\ad_a^q(\partial_1v).\end{aligned}
$$The result is then a consequence of Proposition \ref{auxi}.

 \end{proof}

\begin{proposition}\label{prop:eltriangulo}
 The model of the triangle is given by
\begin{equation}\label{triangulo}
\lasu_2=(\libc(\Delta^2),\partial)=(\hL (a_0,a_1,a_2,a_{01},a_{12},a_{02},a_{012}),\partial),
\end{equation}
in which $(\libc(a_0, a_1, a_{01}),\partial)$, $(\libc(a_1, a_2, a_{12}),\partial)$ and $(\libc(a_0, a_2, a_{02}),\partial)$ are  LS-intervals and where their inclusions in $\lasu_{2}$ give  the coface maps. Moreover, the differential of
$a_{012}$ is defined by
\begin{equation*}
\partial_{a_0}a_{012}=a_{01}*a_{12}*a_{02}^{-1}.
\end{equation*}
\end{proposition}

\begin{proof}
The composition of morphisms as in Theorem \ref{prop:subdivision}
defines the DGL morphism
$\psi\colon (\hL(a,b,x),d)\to (\hL (a_0,a_1,a_2,a_{01},a_{12},a_{02},a_{012}),\partial)$
in which $\psi(a)=\psi(b)=a_{0}$ and $\psi(x)=a_{01}\ast a_{12}\ast a^{-1}_{02}$. Hence,
\begin{eqnarray*}
\partial(a_{01}\ast a_{12}\ast a^{-1}_{02})
&=&
\partial\psi (x) = \psi (\partial x)=
\psi\left(
\ad_{x}(b)+\sum_{k\geq 0}\frac{B_{k}}{k!}\ad^k_{x}(b-a)\right)\\
&=&
{\ad}_{a_{01}\ast a_{12}\ast a^{-1}_{02}}(a_0).
\end{eqnarray*}
Finally observe that the linear part of $\partial(a_{012})$ is $a_{12}-a_{02}+ a_{01}$.
In geometrical terms, this expression
draws the border of $\Delta^2$ starting from the base point~$a_0$.
\end{proof}

\begin{proof}[Proof of \thmref{existencia}.]
We proceed by induction on $n$.  Let $(\lasu_0, \partial)=(\libc(a_0),\partial)$ in which $a_0$ is a Maurer-Cartan element. On the other hand, let  $\lasu_1=( \libc(a_0,a_1,a_{01}),\partial)$ be the LS-interval and let $\lasu_2$, the first one satisfying the inductive statement, be the model of the triangle as in  Proposition \ref{prop:eltriangulo}.

 Assume that $\lasu_{q}=(\libc(\Delta^q),\partial)$ is defined for $q<n$, as in the statement, with $n\geq 3$.
 The condition (3) of  \defref{modelo} defines the differential $\partial$ on each face of $\Delta^n$. In particular, by induction,
 $\partial_{a_{0}}a_{0\ldots n-1}\in  \libc(\Lambda^n_{n})$.
 As $H(\libc(\Lambda_{n}^n),\partial_{a_0})=0$ by Corollary \ref{noseque}(i), there exists $\Gamma\in \libc(\Lambda_{n}^n)$ such that
 $|\Gamma|=n-2$ and $\partial_{a_{0}} a_{0\ldots n-1}=\partial_{a_{0}}\Gamma$.
 We set,
 $$\partial_{a_{0}}a_{0\ldots n}=(-1)^n\left(a_{0\ldots n-1}-\Gamma\right).$$
 By construction, this model satisfies the conditions (1) and (3) of \defref{modelo}. Now denote
 by $\partial_{1}a_{0\ldots n}$ the linear part of  $\partial_{a_{0}}a_{0\ldots n}$ and by $\Gamma_1$ the linear part of $\Gamma$. Since $\partial_1^2= 0$,
 $\partial_{1}\Gamma_{1}=\partial_1(a_{0\dots n-1})$. Let $\omega$ be the
 difference $\omega=(-1)^{n-1}\Gamma_{1}-\sum_{i=0}^{n-1}(-1)^ia_{0\ldots \widehat{i}\ldots n}$. Since $\partial_1(\sum_{i=0}^n (-1)^ia_{0\dots \widehat{i}\dots n}) = 0$,
 $$\partial_{1}\omega=(-1)^{n-1}\partial_1(a_{0\dots n-1}) + (-1)^n \partial_1(a_{0\dots n-1})= 0.$$

 Thus, $\omega$ is a $\partial_1$-cycle of degree $n-2$  in $\libc(\Lambda^n_n)$.
 Hence, there is a linear element, $\gamma$, of degree $n-1$ in $\libc(\Lambda^n_{n})$
 such that $\partial_{1}\gamma=\omega$.
 As $\libc(\Lambda^n_{n})$ is generated by elements of degree $\leq n-2$, we have $\gamma=0$ and $\omega=0$. Therefore, we get the expected linear differential,
 $$\partial_{1}a_{0\ldots n}=(-1)^n a_{0\ldots n-1}+\sum_{i=0}^{n-1} (-1)^ia_{0\ldots \widehat{i}\ldots n-1}.$$
 \end{proof}

We now prove the uniqueness up to isomorphism of the sequences of models.

\begin{theoremb}\label{unicidadcompleta}
Two sequences $\{(\libc(\sdelta^n),\partial)\}_{n\ge 0}$ and $\{(\libc(\sdelta^n),\partial')\}_{n\ge 0}$ of compatible models of $\mathbf\Delta$  are isomorphic: for    $n\ge 0$,  there are DGL isomorphisms,
$$
\varphi_n\colon(\libc(\sdelta^n),\partial)
\xrightarrow[]{\cong}
 (\libc(\sdelta^n),\partial'),
$$
which commute with the coface maps $\delta_i$, for   $i=0,\ldots,n$,
\begin{equation}\label{equa:petitcarre}
\xymatrix{
(\libc(\sdelta^n),\partial)
\ar[r]^{\varphi_n}_\cong&(\libc(\sdelta^n),\partial')\\
(\libc(\sdelta^{n-1}),\partial) \ar@{_{(}->}[u]^{\delta_i} \ar[r]^{\varphi_{n-1}}_\cong&(\libc(\sdelta^{n-1}),\partial'),\ar@{_{(}->}[u]_{\delta_i}}
\end{equation}
and such that
$\im(\varphi_n-\id)\subset \libc^{\ge 2}(\sdelta^n)$.
\end{theoremb}

\begin{proof} Obviously there is only one choice for $(\libc(\sdelta^0),\partial)$ while $(\libc(\sdelta^1),\partial)$ is also uniquely determined by Theorem \ref{unicols}. Suppose   $n\ge 2$ and $\varphi_m$  defined for all $m<n$.
Then, the condition (3) of \defref{modelo} determines $\varphi_n$ on every generator of $\libc(\Delta^n)$ except on $x=a_{0{\dots}n}$.

We may assume without losing generality  that the sequence $\{(\libc(\sdelta^n),\partial)\}_{n\ge 0}$ is inductive.
Hence, as $\partial_{a_0}x\in\libc(\sfdelta^n)$, the hypothesis $\im(\varphi_i-\id)\subset \libc^{\ge 2}(\sdelta^i)$ for $i<n$,
implies that
$\varphi_{n}\partial_{a_{0}}x-\partial'_{a_{0}}x$
is   a decomposable element. Since this is a $\partial'_{a_{0}}$-cycle, there exists $\omega\in \libc(\Delta^n)$, $|\omega|=n-1$, such that
$\partial'_{a_{0}}\omega=\varphi_{n}\partial_{a_{0}}x-\partial'_{a_{0}}x$.
The linear part, $\omega_{1}$, of $\omega$ is a $\partial_{1}$-cycle and there exists a linear element
$\gamma\in\libc(\Delta^n)$, $|\gamma|=n$, such that $\partial'_{1}\gamma=\omega_{1}$.
By  degree reasons, we have $\gamma=0$ and $\omega_{1}=0$. Thus, we set
$$\varphi_{n}(x)=x+\omega.$$
As $\omega\in\libc^{\geq 2}(\Delta^n)$, the condition
$\im(\varphi_n-\id)\subset \libc^{\ge 2}(\sdelta^n)$ is fulfilled. Moreover, the coface maps being linear, the
square (\ref{equa:petitcarre}) is commutative.
\end{proof}

\begin{example}\label{ejemplo}{\bf The model of the tetrahedron}

\vskip.2cm
Let $(L,\partial)$ be a complete DGL, and  $e_1, \dots , e_n\in L_1$. We form the DGL $(L', \partial')= (\widehat{\mathbb L}(e_i, u_i), \partial')$ where $\partial'(u_i)= 0$ and $\partial'(e_i)= u_i$ and the DGL morphism $\gamma\colon  (L', \partial')\to (L, \partial)$
 defined by $\gamma(e_i) = e_i$ and $\gamma(u_i) = \partial e_i$. Now the product $u_1*\cdots *u_n$ is a linear combination of Lie brackets and in each of them we replace one and only one $u_i$ by the corresponding $e_i$. This defines an element $A_{e_1\dots e_n}\in L'$, and we define $B_{e_1\dots  e_n} = \gamma (A_{e_1\dots  e_n})$. By construction we have
 $$\partial B_{e_1 \dots e_n}= \partial e_1*\dots *\partial e_n,$$
 and its linear part is precisely $\sum_{i=1}^n e_i$.

For   the model  of the tetrahedron, we observe that via  \propref{prop:eltriangulo},
\begin{eqnarray*}
\partial_{a_{0}}B_{a_{012},a_{023},-a_{013}}
&=&
\partial_{a_{0}}a_{012} *
\partial_{a_{0}}a_{023}*
(-\partial_{a_{0}}a_{013})\\
&=&
a_{01}*a_{12}*a_{23}*a_{13}^{-1}*a_{01}^{-1}.
\end{eqnarray*}

On the other hand, a direct computation, using successively
\propref{prop:dadb},
the model of the triangle described in \propref{prop:eltriangulo}
and \propref{prop:expadx}, gives
\begin{eqnarray*}
\partial_{a_{0}}\left(e^{\ad_{a_{01}}}a_{123}\right)
&=&
e^{\ad_{a_{01}}}\partial_{a_{1}}a_{123}\\
&=&
e^{\ad_{a_{01}}} (a_{12}*a_{23}*a_{13}^{-1})\\
&=&
a_{01}*a_{12}*a_{23}*a_{13}^{-1}*a_{01}^{-1}.
\end{eqnarray*}
Finally, define
$$\partial_{a_{0}}a_{0123}=e^{\ad_{a_{01}}}a_{123}-B_{a_{012},a_{023},-a_{013}}.$$
\end{example}
We finish by finding base points in the barycenter of the simplex.
\begin{proposition}\label{prop:latotale}
Given $(\widehat{\mathbb L}(\Delta^n), \partial)$ a symmetric  model of $\Delta^n$, there exists a Maurer-Cartan element $a$  whose linear part is the barycentre of $\Delta^n$ and such that, for $n\ge 2$, $\im\partial_a\subset \widehat\lib(\dot\Delta^n)$.
 \end{proposition}

\begin{proof} Consider the sequence of Maurer-Cartan elements $x_0, \cdots , x_n$ defined by
$$
x_0= a_n,\quad
x_r = \frac{a_{rn}}{n+1} \cG x_{r-1}.
$$
Denote by $u_1$ the linear part of an element $u$. Since for $y\in L_0$ and $u\in \mc(L)$, $(y\cG u)_1= u_1-(dy)_1$, we have $(x_n)_1 = \sum_{i=0}^n \displaystyle\frac{a_i}{n+1}$.

The last part of the statement is a consequence of the behavior of the differential on inductive models.
\end{proof}
%

%%%%%%%%%%%%%%%%%%%%%%%%%%%%%%%%%
\section{Symmetric models of  $\mathbf\Delta$ and the cosimplicial structure}\label{sec:symmetric}

Let $a_{i_0\dots i_p}$ be a generator of $\libc(\Delta^n)$ and $\sigma\in \Sigma_{p+1}$ of signature $\varepsilon_{\sigma}$.
We set
 $$a_{i_{\sigma (0)}\ldots i_{\sigma (p)}}=\varepsilon_\sigma a_{i_0\ldots i_p}.$$
  With this notation, we define an action of the symmetric group  $\Sigma_{n+1}$  on the generators of  $\libc(\Delta^n)$ by the rule,
$$
\sigma\cdot a_{i_0\ldots i_p}=a_{\sigma(i_0)\ldots\sigma(i_p)}.
$$
We extend it to brackets by $\sigma\cdot [a,b]=[\sigma\cdot a,\sigma\cdot b]$
and get an action on $\libc(\Delta^n)$.

 \begin{definition} A  sequence $\{(\libc(\sdelta^n),\partial)\}_{n\ge0}$ of compatible models of $\mathbf\Delta$ is called {\em symmetric} if for the above action each $(\widehat{\mathbb L}(\Delta^n),\partial)$ is a $\Sigma_{n+1}$-DGL, that is, a DGL whose bracket and differential are compatible with the $\Sigma_{n+1}$ action.
 \end{definition}

Since $(\lasu_n, \partial_1)$ is $\Sigma_{n+1}$-DGL, we have

\begin{lemma}\label{lem2}
$$H_q(\widehat{\mathbb L}(\Delta^n)^{\Sigma_{n+1}}, \partial_1) = \left\{
\begin{array}{ll} \mathbb Q[\frac{\sum a_i}{n+1}] \hspace{5mm}\mbox{} & \mbox{\text{\em  if} $q=-1$},\\
\,0 & \mbox{\em if $q\geq 0$.}
\end{array}\right.$$
\end{lemma}

\begin{proof}
This is clear for $n = 0$. In general $H(\widehat{\mathbb L}(\Delta^n)^{\Sigma_{n+1}}, \partial_1)$ injects into $H(\widehat{\mathbb L}(\Delta^n), \partial_1)$ $ = \mathbb L (a_0)$. Suppose $\alpha$ is a symmetric cycle whose image is in the same homology class than $a_0$, then $\alpha$ is homologous to the symmetrization $\overline{a_0}$ of $a_0$,
$$\overline{a_0} = \frac{1}{(n+1)!} \sum_{\sigma\in \Sigma_{n+1}} \sigma \cdot a_0.$$
 \end{proof}

 \begin{theoremb}\label{existenciasim} There exists a sequence  of symmetric models of $\mathbf\Delta$.
\end{theoremb}

\begin{proof} The model of $\Delta^0$  is trivally symmetric.
Let $\sigma$ be the generator of $\Sigma_2$ and observe directly from (\ref{equa:dxLSinterval}) and (\ref{dxa}) that in the LS-interval, $\lasu_1=(\libc(a_0,a_1,a_{01}),\partial)$,
the morphism defined by  $\sigma (a_0)= a_1$, $\sigma (a_1)= a_0$ and $\sigma (a_{01})= -a_{01}$ commutes with $\partial$ .
We deduce then directly that   $\lasu_1$ is symmetric.

Assume $(\widehat{\mathbb L}(\Delta^{n-1}),\partial)$ to be symmetric with $n\ge 2$ and observe that, by construction $(\lasu(\dot\Delta^n), \partial)$ is a $\Sigma_{n+1}$-DGL. It remains only to define $d(a_{0\cdots n})$ in order  that $\sigma d(a_{0\cdots n}) = d\sigma (a_{0\cdots n})$ for all permutation $\sigma$.

Write $x= a_{0\cdots n}$ and $\partial = \partial_1+ \partial_2+\cdots$ where $\partial_q$ increases the length of the brackets by $q-1$.
By definition $\partial_1x = \sum_{i=0}^n (-1)^i a_{0\dots \widehat{i}\ldots n}$
  is symmetric, and we suppose by induction that, for all $r<q$, the elements $\partial_rx $ have been defined,
are symmetric and satisfy
$$\sum_{i+j=r+1} \partial_i\partial_jx = 0.$$
Then   $\sum_{i=2}^{q}\partial_i\partial_{q+1-i}x$ is a decomposable symmetric $\partial_1$-cycle.
By Lemma \ref{lem2} there exists a symmetric element $\omega_{q}$ such that
$$\sum_{i=2}^{q}\partial_i\partial_{q+1-i}x=\partial_{1}\omega_{q}.$$
We set $\partial_qx = \omega_q$ and check easily that the induction step is attained.
 \end{proof}

\begin{theoremb}\label{cosimplicial} Any   sequence $\{\lasu_n\}_{n\ge 0}$ of  models of $\mathbf\Delta$ admits a  cosimplicial DGL structure for which the cofaces are the usual ones.
\end{theoremb}

\begin{proof}   Assume first that the models $\lasu_{n}$ are symmetric. For $0\leq i\leq n$, denote as usual by $\sigma_i   \colon \{0,\dots, n+1\}\to \{0, \dots, n\}$ the morphism defined by
$$\sigma_i(j) = \begin{cases} \,\, j&\text{if $j\le i$},\\ j-1&\text{if $j> i$}.
\end{cases}
$$
Then,    we define a  DGL morphism, denoted in the same way,
$$
\sigma_i (a_{\ell_0\dots \ell_q}) = \left\{\begin{array}{ll}
a_{\sigma_i (\ell_0)\dots  \sigma_i (\ell_q)}   & \mbox{if } \sigma_i (\ell_0)<\dots <\sigma_i (\ell_q),
\\
\qquad 0 & \mbox{otherwise.}
\end{array}\right.
$$
When the elements $\sigma_i (\ell_j)$ are different, then $\sigma_i$ extends to an element of $\Sigma_{n+2}$ and $\partial \sigma_i (a_{\ell_0\dots \ell_q}) = \sigma_i \partial (a_{\ell_0\dots \ell_q}).$

In the case the sequence $\ell_0,\dots , \ell_q$ contains the elements $i$ and $i+1$, then denote by $\tau$ the permutation $\tau = (i, i+1)$. We have $\sigma_i\circ \tau = \sigma_i$. Thus,
$$\sigma_i \partial (a_{\ell_0\dots \ell_q}) = \sigma_i \tau \partial (a_{\ell_0\dots \ell_q}) = -\sigma_i \partial (a_{\ell_0\dots \ell_q}).$$
Therefore $\sigma_i\partial (a_{\ell_0\dots \ell_q}) = 0 = \partial \sigma_i(a_{\ell_0\dots \ell_q}),$
and thus  the $\sigma_i$'s are DGL morphisms. The cosimplicial identities are trivially satisfied.

  Now let $\{\lasu'_n\}_{n\ge 0}$ be another sequence of compatible models. By Theorem \ref{unicidadcompleta} we have compatible isomorphisms
  $$\varphi_n \colon \lasu_n' \stackrel{\cong}{\longrightarrow} \lasu_n.$$
We  define then the codegeneracies as
$\varphi_n^{-1}\sigma_i\varphi_n$.
Since the $\varphi_n$'s commute  with the cofaces, the cosimplicial identities are also satisfied in this case.
\end{proof}

%%%%%%%%%%%%%%%%%%%%%%%%%%%%
\section{\bf  The realization functor and its adjoint}\label{sec:larealizacion}

Based on a sequence $\lasu_\bullet$ of compatible models of $\mathbf\Delta$ with the cosimplicial structure given by Theorem \ref{cosimplicial}, we define  a pair of adjoint functors,
$$\xymatrix{ \catss& \catdgl \ar@<1ex>[l]^(.40){\langle\,\cdot\,\rangle}
\ar@<1ex>[l];[]^(.55){\lasu}\\}
$$
between the categories of simplicial sets and complete DGL's.
\begin{definition}\label{def:realization}
Let $L\in \catdgl$.
The \emph{realization} of  $L$ is   the simplicial set,
$$
\langle L\rangle_\bullet=\catdgl(\lasu_\bullet,L).
$$
\end{definition}

On the other hand, let $\underline{\mathbf\Delta}$ be the category whose objects are the sets $[n]=\{0,\dots,n\}$, $n\ge 0$, and whose morphisms are monotone maps. Now, let $I\colon \underline{\mathbf \Delta } \to \catss  $ the functor that associates to $[n]$ the simplicial set $\underline\Delta^n$ whose $p$-simplices are  the sequences $0\leq i_0\leq \cdots \leq i_p\leq n$. Observe that, by construction, $\lasu_\bullet$ is a functor from $\underline{\mathbf \Delta}$ to $\catdgl$.

\begin{definition}  The functor \emph{model} $\lasu\colon \catss\to \catdgl$ is defined as the    left Kan extension of $\lasu_{\bullet}$ along $I$,
$$\xymatrix{
\underline{\mathbf \Delta}\ar[r]^-{I}\ar[d]_{\lasu_{\bullet}}&
\catss\ar[dl]^-{\mbox{}\hspace{2mm}\lasu={\text{Lan}}_{I}\lasu_{\bullet}}\\
\catdgl
}$$
\end{definition}

\noindent The DGL $\lasu (K)$ is thus the colimit of $\lasu_\bullet$ over the comma category $I\downarrow K$,
$$\lasu (K)={\text{Lan}}_{I}\lasu_{\bullet}(K)= \varinjlim_{f\colon \underline\Delta^n \to K} \,\, \lasu_n.$$
For simplicity, we write
$$
\lasu(K)=\varinjlim_K\lasu_{\bullet}.
$$
and refer to it as the   $\lasu$-{\em model} of the simplicial set $K$.

\vspace{3mm}In the case $K$ is a finite simplicial complex, then $K\subset \Delta^n$ for some $n$, and $\lasu (K)$ is trivially isomorphic to the  complete sub DGL $(\widehat{\mathbb L}(V),\partial) \subset \lasu_n$ where $(V,\partial_1)$ is the desuspension of the chain complex of $K$.

\begin{theoremb} The functors $\lasu$ and $\langle\,\cdot\,\rangle$ are adjoint.
More precisely, for any simplicial set $K$ and any  complete differential graded Lie algebra $L$, there is a bijection,
$$\catss(K, \langle L\rangle) \cong \catdgl(\lasu(K), L).$$
\end{theoremb}

\begin{proof}
The result follows from classical properties of commutation of  limits with hom functors, i.e.,
$$
\begin{aligned}
\catdgl(\lasu(K), L) &= \catdgl(\varinjlim_K\lasu_n, L) = \varprojlim_K\,\catdgl(\lasu_n, L)\\
 &=\varprojlim_K \langle L\rangle_n = \varprojlim_K\catss(\underline\Delta^n, \langle L\rangle)\\
&= \catss(\varinjlim_K\, \underline\Delta^n, \langle L\rangle) = \catss(K, \langle L\rangle).
\end{aligned}$$
\end{proof}

We now  interpret the homotopy groups  of the realization of a DGL and its path component.

\begin{proposition}\label{pi0}
For any DGL, $(L, \partial)$, there is a natural bijection $\pi_0\langle L\rangle \cong \widetilde{{\MC}}(L)$.
\end{proposition}

 \begin{proof}
 By \cite[Proposition 3.1]{BM}, two Maurer-Cartan elements $z_0,z_1\in\mc(L)$ are  gauge equivalent if there is a map
 $\varphi\colon {\lasu}_1= (\widehat{\mathbb L}(a,b,x),\partial)\to L$
 with $\varphi (a) = z_0$ and $\varphi (b) = z_1$.
 By \defref{def:realization}, $\langle L\rangle_0$ is the set of Maurer-Cartan elements of $L$, and $\langle L\rangle_1$ is the set of DGL morphisms from the LS-interval $\lasu_1$ to   $L$.  This implies the result.
 \end{proof}

 \begin{proposition}\label{componentes} Let $(L,\partial)$ be a non negatively graded   DGL.
 Then,   $\langle L\rangle$ is a connected simplicial set  and there are natural group isomorphisms
$$
\pi_n\langle L\rangle\cong H_{n-1}(L,d),\qquad n\ge 1,
$$
in which $H_0(L,d)$ is considered with the group structure given by the Baker-Campbell-Hausdorff product.

\end{proposition}

\begin{proof} By \propref{pi0}, $\langle L\rangle$ is connected.
The coface maps, $\delta_{j}\colon \lasu_{n-1}\to \lasu_{n}$, induce the face maps
$$d_{i}=\catdgl(\delta_{j},L)\colon \langle L\rangle_{n}\to \langle L\rangle_{n-1}.$$
We denote
$\ker d_j=\{f \colon (\lasu_n, \partial) \to (L,\partial) \mid d_jf=0\}$.
 Recall that
$$
\pi_n \langle L\rangle=\cap_{i=0}^{n}\ker d_i/\sim
$$
where $f\sim g$ if  there is $h\in \langle L\rangle_{n+1}$ such that $d_nh=f$, $d_{n+1}h=g$ and $d_ih= 0$ for $i<n$. We denote by $\overline f$ the element of $\pi_n \langle L\rangle$ represented by $f$. Define,
$$
\varphi\colon \pi_n \langle L\rangle\stackrel\cong\longrightarrow H_{n-1}(L),\qquad \varphi(\overline f)=[f(a_{0\dots n})],
$$
and observe that, for any $\overline f\in\pi_n\langle L\rangle$, the  morphism $f$ vanishes in any  $p$-simplex of $\Delta^n$, with $0\le p< n$. Hence, it is uniquely determined by $f(a_{0\dots n})$.  Straightforward computations show that $\varphi$ is a well defined isomorphism for $n\ge 2$ and a bijection for $n=1$.

To check that for $n=1$ this bijection is in fact an isomorphism of groups choose $\alpha,\beta\in\pi_1\langle L\rangle$, and consider $h\in\langle L\rangle_2=\catdgl(\lasu_2,L)$ given by $$
h(a_{01})=g(a_{01}), \quad h(a_{12})=f(a_{01}),\quad h(a_{02})=f(a_{01})*g(a_{01}),\quad h(a_{012})=0,
$$
being $f,g\in \langle L\rangle_1=\catdgl(\lasu_1,L)$ representing $\alpha$ and $\beta$ respectively. Note that, since $L$ is non negatively graded the image of any $0$-simplex vanishes for every morphism in $\langle L\rangle$.

Now, in view of the model of $\Delta^2$ given in Proposition \ref{prop:eltriangulo}, $h$ is a well defined morphism for which $d_0h=f$ and $d_2 h=g$. Hence, by definition of the product in $\pi_1\langle L\rangle$, $\alpha\cdot\beta$ is represented by $d_1h$. Finally,
$$
\varphi(\alpha\cdot\beta)=d_1h(a_{01})=h\delta_1(a_{02})=h(a_{02})=f(a_{01})*g(a_{01})=\varphi(\alpha)*\varphi(\beta).
$$

\end{proof}

For any differential graded Lie algebra $L$ and any Maurer-Cartan element $z\in \MC(L)$ consider the {\em localization of $L$ at $z$} which is the DGL,
$$
L^{(z)}=(L,\partial_z)/(L_{<0}\oplus M)\cong L_{>0}\oplus(L\cap \ker \partial_z)
$$
where $M$ is a complement of $\ker\partial_z$ in $L_0$. Observe that  the injection of $(L^{(z)},d_z)\hookrightarrow (L,d_z)$ induces an isomorphism in homology in degrees $\geq 0$.

\begin{theoremb}\label{localizar} $\langle L\rangle\simeq \dot\cup_{z\in\widetilde\MC(L)} \langle L^{(z)}\rangle$.
\end{theoremb}

\begin{proof}

As we know, the components of $\langle L\rangle$ are identified with $\widetilde \MC(L)$. Via this identification,   the component of a given $z\in \widetilde\MC(L)$  is of the same homotopy type as the reduced simplicial set which we denote by
$
\langle L\rangle_z
$
whose $n$-simplices are the DGL morphisms $f\colon\lasu_n\to L$ such that $f(a_i)=z$ for any $0$-simplex $a_i$, $i=0,\dots,n$.

The simplicial set  $\langle L\rangle_z$ has only one $0$-simplex $\overline{z} \colon \lasu_0 \to L$ and its degeneracies are the maps $\overline{z} : \lasu_n \to L$ such that $\overline{z}(a_i)= z$ for all $i$ and which vanish on all generators of non-negative degrees.
 Observe that, for any $n\ge 1$,
$\pi_n(\langle L\rangle_z, \overline{z}) $ is the quotient space $E_n/\sim$, where $E_n$ denotes the set of DGL morphisms $f\colon \lasu_n\to L$ such that $d_if = \overline{z}$ for all $i$. When $f\in E_n$, $f(a_{0\dots n})$ is a $\partial_z$-cycle which defines an isomorphism
$\pi_n(\langle L\rangle_z, \overline{z}) \cong H_{n-1}(L, \partial_z)\,$
that is in turn induced by the simplicial set weak equivalence
 $$\psi\colon \langle L \rangle_z \stackrel{\simeq}{\longrightarrow}\langle L^{(z)}\rangle,\quad \psi(f) (a_i) = 0,\quad \psi (f)(a_{i_0\dots i_q}) = f(a_{i_0\dots i_q})\,\,\text{for $q>0$.}
  $$
\end{proof}

\section{The $\lasu$-model of $K$}

 We first prove that the homology of ${\lasu}(K)$ for the differential $\partial_a$ depends only on the component of $a$.

 \begin{proposition}\label{prop:component}
  Let $K$ be a finite simplicial complex and let $a$ be a $0$-simplex. Denote by  $K_a$  the component of $a$ in $K$ and denote also by $a$ the corresponding generator of degree $-1$ in $\lasu (K)$. Then, the injection $({\lasu}(K_a), \partial_a) \stackrel{\simeq}{\to} ({\lasu}(K),\partial_a)$ is a quasi-isomorphism.
 \end{proposition}

 \begin{proof}  Write $K = K_{a}\amalg K'$, ${\lasu}(K_a) = \widehat{\mathbb L}(V\oplus \mathbb Q a)$
 and ${\lasu}(K') = \widehat{\mathbb L}(W)$. Recall (see for instance \cite[Proposition 6.2.(7)]{tan}) that  the  ideal $I$ in $\mathbb L(V\oplus W\oplus \mathbb Q a)$ generated by $V\oplus W$  is the free DGL $ I= \mathbb L(\overline{V}\oplus \overline{W})$,  where
 $$\overline{V} = \{ \ad^n_a(v_i), n\geq 0, \quad \mbox{with $(v_{i})$   a basis of $V$}\},$$
  $$\overline{W} = \{ \ad^n_a(w_j), n\geq 0, \quad \mbox{with $(w_j)$   a basis of $W$}\}.$$

Since $V\oplus W$ is finite dimensional  and $a$ has degree $-1$, the kernel $J$ of the projection $(\widehat{\mathbb L}(V\oplus W\oplus \mathbb Q a), \partial_a) \to (\mathbb L(a), \partial_a)$ is the DGL $(\widehat{\mathbb L}(\overline{V}\oplus \overline{W}), \partial)$. Denote by $d_1$ the linear part of the differential $\partial$ in $J$.
We claim that  for $w\in W$ and $q$ even, $d_1 \ad_a^q(w)= \ad_a^{q+1}(w)+ (-1)^q \ad_a^q(d_1w)$.
Suppose this is true for $q-2$. Then, since $\ad_a^q(v) = \frac{1}{2}\bigl[[a,a], \ad_a^{q-2}(v)\bigr]$ we have
 $$
 \begin{aligned}d_1 \ad_a^q(v) &= \frac{1}{2} d_1 \bigl[[a,a], \ad_a^{q-2}(v)\bigr] = \frac{1}{2}\bigl[[a,a], \ad_a^{q-1}(v)+ (-1)^q \ad_a^{q-2}(d_1v)\bigr]\\
 &= \ad_a^{q+1}(v)+ (-1)^q\ad_a^q(d_1v).
 \end{aligned}$$

Therefore the homology $H(\overline{W}, \partial_1) = 0$
and the injection $(\widehat{\mathbb L}(\overline{V}), \partial_a)\stackrel{\simeq}{\to} (\lasu (K),\partial_a)$ is a quasi-isomorphism. This implies the result.

\end{proof}

 \begin{definition}\label{defi:modelos}
 Let $K$ be a simplicial set.
 A \emph{model of $K$} is a free complete differential Lie algebra $(\widehat{\mathbb L}(Z),\partial)$
 connected to $\lasu(K)$ by a sequence of quasi-isomorphisms.
 Such model is called \emph{minimal} if the linear part of the differential is null, i.e., $\partial_{1}=0$.
 \end{definition}

We precise the behaviour of some models in the connected case.

\medskip
 Let $K$ be a finite connected simplicial complex.
Then $(\lasu(K), \partial)= (\widehat{\mathbb L}(V), \partial)$,
 where $(V, \partial_1)$ is the desuspension of the chain complex of $K$.
 Let $a_0, \ldots , a_n$ be the $0$-simplices of $K$ viewed as generators  of $V_{-1}$.
 We write $W = \oplus_{q\geq -1} W_q$ with $W_q= V_q$ for $q\geq 0$ and $W_{-1}$
  the vector space generated by the elements $a_i-a_0$.

\begin{proposition}\label{prop:minimodel}
Let $K$ be a finite connected simplicial complex. With the previous notation, the differential
$\partial$ in $\lasu(K)$ can be chosen so that $\partial_{a_0}(W)\subset \widehat{\mathbb L}(W)$
and there is a quasi-isomorphism,
\begin{equation*}
(\lasu(K), \partial_{a_0}) \stackrel{\simeq}{\longrightarrow} (\widehat{\mathbb L}(W), \partial_{a_0}).
\end{equation*}
Moreover, if we denote by  $(\lasu(K)/(a_0), \overline{\partial})$  the quotient of $(\lasu(K), \partial)$
 by the ideal generated by $a_0$, then there is  an isomorphism,
\begin{equation*}
(\lasu(K)/(a_0), \overline{\partial} ) \stackrel{\cong}{\longrightarrow} (\widehat{\mathbb L}(W), \partial_{a_0}).
\end{equation*}
\end{proposition}

\begin{proof}
We prove by induction on $n$ that $\partial_{a_0}(W_{n})\subset \widehat{\mathbb L}(W)$. By the naturality of the differential in $\lasu(K)$ we have only to prove it when $K = \Delta^{n+1}$.
 It is obvious for $n=-1$ and  directly deduced for $n=0$ from  the definition of the differential of the LS-interval.
Suppose that $\partial_{a_0}(W_q)\subset \widehat{\mathbb L}(W)$ for  $q<n-1$ and $n\geq 2$.
From the induction hypothesis, we get
$$(\widehat{\mathbb L}(\dot\Delta^n), \partial_{a_0}) = (\mathbb L(a_0), \partial_{a_0})\,\, \widehat{\amalg}\,\,  (\widehat{\mathbb L}(Z), \partial_{a_0}),$$
where $Z$ is the graded vector space generated by the elements $\{a_i-a_0\mid i>0\}$,
together with the generators of degree $\geq 0$. Here, $\widehat\amalg$ denotes the coproduct in the category of DGL's which is obtained by completing the non-complete coproduct. Since $H(\mathbb L (a_0), \partial_{a_0})= 0$,
the inclusion of the second factor is a quasi-isomorphism,
\begin{equation}\label{equa:lea0}
  (\widehat{\mathbb L}(Z), \partial_{a_0})\stackrel{\simeq}{\longrightarrow}
(\mathbb L(a_0), \partial_{a_0})\,\, \widehat{\amalg}\,\,  (\widehat{\mathbb L}(Z), \partial_{a_0}).
\end{equation}
On the other hand, in view of \corref{noseque},
$$H(\widehat{\mathbb L}(\dot\Delta^n), \partial_{a_0}) = \mathbb L(u),$$
where $\vert u\vert = n-2$.
With (\ref{equa:lea0}), we can choose $u\in \widehat{\mathbb L}(Z)$, and set $\partial_{a_0}(a_{0\dots n}) = u$.
Since $W= Z\oplus \mathbb Q a_{0\dots n}$, the inclusion $\partial_{a_0}(W)\subset \widehat{\mathbb L}(W)$ is established. Then,
it follows that
$$  (\lasu (K), \partial_{a_0})\stackrel{\simeq}{\longrightarrow}
(\mathbb L(a_0), \partial_{a_0})\,\, \widehat{\amalg}\,\,  (\widehat{\mathbb L}(W), \partial_{a_0}).$$
By Lemma \ref{thel}  this implies that $(\lasu (K), \partial_{a_0})$ is quasi-isomorphic to $(\widehat{\mathbb L}(W), \partial_{a_0})$. Moreover on the quotient $\lasu (K)/(a_0)$ the induced differentials $\overline{\partial}$ and $\overline{\partial_{a_0}}$ coincide. This implies the
  isomorphism of the statement.
\end{proof}

\begin{corollary}\label{cor:minimodel}
Let $K$ be a connected finite simplicial complex. Then,  $(\lasu(K), \partial_{a_0}) $ has a minimal model of the form
$  (\widehat{\mathbb L}(R), \partial)$, with $R= R_{\geq 0}\cong s^{-1}\widetilde{H}(K;\Q)$.
\end{corollary}

\begin{proof}
Denote by $T$ a maximal tree contained in the graph formed by the 0-simplices and the 1-simplices of $K$.
Since $K$ is connected, this tree contains all the vertices.
Moreover, in $\lasu(K)$, the ideal $\cI$, generated by the $\{a_i-a_0\mid  i>0\}$
together with  the generators of dimension 0 corresponding to the edges of $T$,  is acyclic.
Therefore, the quotient $\lasu(K)/\cI$ is quasi-isomorphic to a DGL of the form
$$(\mathbb L (a_0), \partial_{a_0}) \,\,  \widehat{\amalg}\,\, (\widehat{\mathbb L}(R), \partial),$$
with $R = R_{\geq 0}$.

Now, suppose that for some $x\in R$ we have $\partial_1x= y\neq 0$. Then, replace $y$ by $\partial x$, and by Proposition \ref{auxi}, taking the quotient by the ideal generated by $x$ and $y$ gives a quasi-isomorphism $(\widehat{\mathbb L}(R), \partial)\stackrel{\simeq}{\to} (\widehat{\mathbb L}(R)/(x,y), \overline{\partial})$. We continue in the same way and finally obtain a minimal complete free DGL quasi-isomorphic to $(\lasu (K), \partial_{a_0})$.

\end{proof}

\begin{example} If $K$ is a connected finite simplicial complex of dimension 1, then $(\lasu (K),\partial_a)\simeq (\widehat{\mathbb L}(V),0)$ with $V = s^{-1} {H}_1(K)$.
\end{example}

\begin{example}\label{exem:elnano}
The minimal model  of a finite simplicial set $K$ of the homotopy type of
$S^1\vee S^2$ is of the form
$(\widehat{\L}(a,b),0)$,
with $|a|=0$ and $|b|=1$.
In particular, $H(\lasu_{K})$
is the completion of $\L(a,b)$ for the adjoint action of $a$ on $b$.
\end{example}

%%%%%%%%%%%%%%%%%%%%%%%%%%%
\section{Differential graded Lie coalgebras and the Transfer Theorem}

In this, and in the next section, we show how the homotopy transfer theorem and the structure of free Lie coalgebras let us also associate a DGL to a finite simplicial complex $K$. The uniqueness theorem (Theorem 2.8) will then identify this DGL with our construction $\lasu(K)$. Following this approach we may understand the relation between the combinatorial properties of $\lasu(K)$ with the algebra of PL-forms on $K$.

\vspace{2mm}
Recall that a graded Lie coalgebra is a graded vector space, $V$, with a comultiplication $\Delta\colon V\to V\otimes V$ that satisfies two properties,
$$(1+\tau)\circ \Delta = 0\quad \mbox{and} \quad
(1+\sigma + \sigma^2)\circ (1\otimes \Delta)\circ \Delta= 0.$$ Here, $\tau {\colon} V\otimes V\to V\otimes V$ is the graded permutation and $\sigma{\colon} V\otimes V\otimes V\to V\otimes V\otimes V$ is the graded cyclic permutation \cite{Fr,Mi}.

Any graded coalgebra $(C, \Delta)$ admits a Lie coalgebra structure defined by $\Delta_L = \Delta-\tau \circ\Delta$.
If $T^c(V)$ denotes the tensor coalgebra on a graded vector space $V$, then $\Delta_L$ induces a Lie coalgebra structure
on the indecomposables for the shuffle product, i.e., on the quotient of $T^c(V)$ by the shuffle products.
This quotient is called the {\em free Lie coalgebra on $V$}  and is denoted by $\Lc(V)$.
We denote by {\bf DGLC} the category of  differential graded Lie coalgebras and by
{\bf CDGC} the category of differential connected  graded cocommutative coalgebras.
If $\mbox{\bf Vect}$ denotes the category of graded vector spaces, there is a natural isomorphism
$$\mbox{\bf Vect}(\Gamma, V) = \mbox{\bf DGLC} (\Gamma, \Lc(V))$$
when $V$ is graded vector space and $\Gamma$ a graded Lie coalgebra \cite[\S 4.2.1]{Fr}.
The graded universal coenveloping coalgebra of a graded Lie coalgebra
 $L$ is a coalgebra, denoted $U^cL$ with the property that for any graded coalgebra $C$, we have
 $$\mbox{\bf DGLC}(C, L) \cong \mbox{\bf CDGC} (C, U^cL).$$
 In particular, there is a coalgebra isomorphism,$U^c(\Lc(V)) \cong T^c(V)$.
If $(T^c(V),d)$ is a differential graded coalgebra such that, for every $n$,
the part $d_n\colon \otimes^n V\to V$ of the differential vanishes on the $(p,q)$-shuffles with $p+q=n$,
then $d$ makes of $(\Lc(V),d)$ a differential graded Lie coalgebra.

Clearly, the dual of a differential Lie coalgebra is a differential Lie algebra. Moreover, the following holds.

\begin{lemma}\label{gruno}
For any finite type vector space, $V$, we have
 $(\Lc(V))^{\sharp}= \widehat{\mathbb L}(V^{\sharp})$.
\end{lemma}

\begin{proof}
From $\Lc(V)= \varinjlim_n (\Lc)^{\leq n} (V)$, we deduce,
$$\begin{aligned}
 \mbox{Hom}(\,\displaystyle\varinjlim_n (\Lc)^{\leq n}(V), \mathbb Q) &=
 \displaystyle\varprojlim_n \,(\mbox{Hom}(\Lc)^{\leq n}(V), \mathbb Q)\\
 & = \displaystyle\varprojlim_n \,\mathbb L(V^{\sharp})/\mathbb L^{>n}(V^{\sharp} )
 = \widehat{\mathbb L}(V^{\sharp}).
 \end{aligned}$$
\end{proof}

Let $f {\colon} (T^c(V),d) \to (T^c(W),d)$ be a morphism of differential graded coalgebras. We denote by $d_1$ the linear part of the differentials, $d_1(V)\subset V$, $d_1(W)\subset W$, and we write $f = f_0+f_1+ \cdots$ with $f_i \,(T^c)^r(V)\subset (T^c)^{r-i}(W)$.

\begin{proposition}\label{qi} With the above notations, suppose $f_0 {\colon} (V,d_1)\stackrel{\simeq}{\to} (W,d_1)$ is a quasi-isomorphism. Then,
\begin{enumerate}
\item[(i)] $f$ is a quasi-isomorphism.
\item[(ii)] Moreover, if $f$ and the differentials $d$  vanish   on the decomposables for the shuffle product, then  the induced map $f {\colon} ({\mathbb L}^c(V),d)\stackrel{\simeq}{\to}({\mathbb L}^c(W),d)$ is also a quasi-isomorphism.
\end{enumerate}\end{proposition}

\begin{proof}  (i) First of all, by the K\"unneth Theorem,
$$f_0^{\otimes n} {\colon} (T^n(V),d_1)\to (T^n(W),d_1)$$
is a quasi-isomorphism for each  $n$. Therefore, by the five lemma,
$$f_{\leq q}=f_0+\cdots +f_q {\colon} (T^{\leq q}(V),d) \to (T^{\leq q}(W),d)$$
is a quasi-isomorphism for all $q\geq 1$. This implies the surjectivity of $H(f)$. To prove the injectivity suppose $H(f)(\alpha)= 0$. Write $\alpha = [a]$, $a\in T^{\leq q}(V)$ and $f(a)= db$, $b\in T^{\leq r}(W)$, for some $r\geq q$. Then since $f_{\leq r}$ is a quasi-isomorphism, $\alpha = 0$.

(ii) Recall that ${\mathbb L}^c(V)$ is a retract of $T^c(V)$,   obtained by identifying $\Lc(sV)$ with the image of the first Eulerian idempotent $e^{(1)}$, i.e.,
see \cite[Theorem 4.2.3]{Fr},
$$\Lc(sV) = e^{(1)} T^c(sV).$$
Since this retraction is natural with respect to linear maps, it commutes with the differentials $d_1$ and with $f_0$. Therefore,
$${\mathbb L}^c(f_0) {\colon} ({\mathbb L}^c(V), d_1) \to ({\mathbb L}^c(W), d_1)$$ is a retract of $T^c(f_0)$ and is thus a quasi-isomorphism. We deduce in the same way that $f {\colon} ({\mathbb L}^c(V),d)\stackrel{\simeq}{\to}({\mathbb L}^c(W),d)$  is a quasi-isomorphism.
\end{proof}

\medskip
The Quillen correspondence ${\mathcal L}$ and ${\mathcal C}$
(generalized by Neisendorfer in \cite{nei}) between $\catdgl$ and
$\CDGC$ dualizes in a pair of adjoint functors between $\catdglco$
and the category $\catcdga$ of augmented commutative differential graded algebras,
that we can summarize in the following commutative diagram \cite[Theorem 4.17]{SW},
$$
\xymatrix{
\CDGC
\ar@/^/[rr]^{\mathcal L}
\ar[d]^{(-)^{\sharp}}
&&
 \catdgl\ar@/^/[ll]^{\mathcal C}\\
\catcdga
\ar@/_{}/[rr]_{\mathcal E}
&&
 \catdglco. \ar[u]_{(-)^{\sharp}}
  \ar@/_/[ll]_{\mathcal A}
 }$$
Moreover, the adjunction maps, $\cL\cC(L)\stackrel{\simeq}{\to} L$ and $C\stackrel{\simeq}{\to} \cC\cL(C)$,
are quasi-isomorphisms, see \cite[Proposition 4.1]{nei}.
In the same way, we have also quasi-isomorphisms,
$\cA\cE(A)\xrightarrow[]{\simeq} A$ and $E\xrightarrow[]{\simeq} \cE\cA(E)$.

The functor ${\mathcal A}$ is defined as follows. Let $(E,d)$ be a differential graded Lie coalgebra. Then, ${\mathcal A}(E)= (\land sE,D)$ with
$$D(sx) = \frac{1}{2} \sum_i (-1)^{\vert x_i\vert} \,\, sx_i\land sx_i' - sdx,$$
with $\Delta x= \sum_i x_i\otimes x_i'$. In particular, if $(\land Z,d)$ is the minimal model of ${\mathcal A}(E)$, then $Z\cong sH(E,d)$.

\vspace{3mm}

Recall now that an \emph{$A_\infty$-algebra structure} on a graded vector space $A$ consists in a sequence of linear maps $m_n {\colon} A^{\otimes n}\to A$, $n\geq 1$,  of degree $2-n$ such that for all $r,s,t\geq 1$
$$\sum_{r+s+t= n} (-1)^{r+st}   m_{r+t+1} (id^{\otimes r} \otimes m_s\otimes id^{\otimes t})= 0.$$
Now   to each $m_n$ is associated a canonical    map of degree $1$, $d_n {\colon} (sA)^{\otimes n} \to sA$. The set of $d_n$ induces a coderivation $d$ on the coalgebra $T^c(sA)$. It is well known that $A$ is an $A_\infty$-algebra if and only if $(T^c(sA),d)$ is a differential graded coalgebra. Moreover, in the case $A$ is a differential graded algebra, we recover  the bar construction on $A$, $(T^c(sA),d)$.

A \emph{morphism of $A_\infty$-algebras} $f{\colon} A\to B$ consists of a sequence of morphisms $f_n {\colon} A^{\otimes n}\to B$ such that the induced morphism
$F {\colon} (T^c(sA),d)\to (T^c(sB),d)$ is a morphism of differential graded coalgebras.

An $A_\infty$-algebra structure on $A$ is \emph{commutative} if the differential $d$ vanishes on the shuffle products in
${T^c}(sA)$.
In this case, $d$ induces  a differential on the quotient $\Lc(sA)$,
and we have a differential graded Lie coalgebra, denoted $(\Lc(sA),d)$.
When $f$ is a morphism between commutative $A_\infty$-algebras, that preserves the shuffle products,
then $f$ restricts to a morphism of Lie coalgebras, $(\Lc(sA), d) \to (\Lc(sA'), d')$.

A \emph{unital $A_\infty$-algebra} is an $A_\infty$-algebra endowed with an element $1_A$ of degree $0$ such that $m_1(1_A)= 0$, $m_2(1_A,a)= a=m_2(a, 1_A)$ for all $a\in A$ and such that for all $i>2$ an all $a_1, \dots , a_i\in A$, the product $m_i(a_1,\dots , a_i) $ vanishes if one of the $a_i$ equals $1_A$. A morphism of $A_\infty$-algebras, $f{\colon} A\to B$ is unital if $f_1(1_A)= 1_B$ and for $n>1$ $f_n(a_1, \dots , a_n)= 0$ if one of the $a_i$ equals $1_A$. A unital $A_\infty$-algebra is \emph{augmented} if it is endowed with a unital morphism $\varepsilon {\colon} A\to \mathbb Q$ such that $\varepsilon (1_A)= 1$. The augmentation ideal $\overline{A}= \mbox{ker}\, \varepsilon$ inherits then  an $A_\infty$-algebra structure.
A morphism of augmented $A_\infty$-algebras is a unital morphism $f {\colon} A\to B$ such that $\varepsilon_B\circ f= \varepsilon_A$.  It induces a morphism between the augmentation ideals.

In   particular, the functor ${\mathcal E}$ is extended to any augmented   $A_\infty$-algebra $A$ as the graded Lie coalgebra ${\mathcal E} (A)=(\Lc(s\overline{A}),d)$.

\medskip
In our setting, a {\em transfer CDGA diagram} is a diagram of the form
$$
\xymatrix{{\fracd}\colon& \ar@(ul,dl)@<-5.5ex>[]_\phi  & (A,d) \ar@<0.75ex>[r]^-p & (V,d) \ar@<0.75ex>[l]^-i }
$$
where $(A,d)$ is an augmented   CDGA, $(V,d)$ is an augmented (cochain) differential graded vector space, $p$ and $i$ are quasi-isomorphisms preserving the augmentations, $pi = \id_V$, and $\phi$ is a chain homotopy between $\id_A $ and $ip$, i.e., $\phi d+d\phi =\id_A- ip$, which satisfies  $\phi i=p\phi=\phi^2=0$.

\begin{theoremb}\label{thm:qicochain}
 The diagram ${\fracd}$ induces   quasi-isomorphisms of differential graded Lie coalgebras,
$$I {\colon} (\Lc(sV),D)\stackrel{\simeq}{\longrightarrow} (\Lc(sA),D),$$
where $I_1=si$ and $D_1=sd$ and
 $$I {\colon} (\Lc(s\overline{V}),D)\stackrel{\simeq}{\longrightarrow} (\Lc(s\overline{A}),D).$$
Moreover, since $(A,d)$ is commutative, the differential graded Lie coalgebra $(\Lc(s\overline{A}),D)$
 is quasi-isomorphic to ${\mathcal E}(A,d)$.
\end{theoremb}

\begin{proof}
In \cite{KS}, Kontsevich and Soibelman proves that there is an augmented $A_\infty${-}structure on $V$, $(\ov{T}^c(sV),D)$,
and  a morphism of $A_\infty$-algebras, $I {\colon} ( {T}^c(sV),D) \to ( {T}^c(sA),D)$,
such that the linear part of $I$ and the linear part  of the differentials $D$ are the suspensions of the original ones.
Moreover, by construction, $I (s\overline{V}) \subset T^c(s\overline{A})$.
The theorem follows then directly from Proposition \ref{qi}.

\end{proof}

\section{\bf The cosimplicial structure via a transfer}\label{sec:transfer}

Let $K$ be a finite simplicial complex. We denote by $C^*(K)$   the dual of the chain complex   of
$K$, with rational coefficients, and by $A_{PL}(K)$ the algebra of rational PL-forms on  $K$, \cite{FHT,Su}.

\begin{theoremb}\label{quehago}{\em {\cite{Du1,Du2,getz}} } There is a CDGA  transfer diagram of the form,
$$
\xymatrix{
{\fracd}\colon
&
 \ar@(ul,dl)@<-5.5ex>[]_{\phi}
 &
  A_{PL}(K) \ar@<0.75ex>[r]^-{p}
  &
   {C^*(K).} \ar@<0.75ex>[l]^-{i} }
$$
In particular there is
 a simplicial CDGA transfer diagram,
$$
\xymatrix{
{\fracd_\bullet}\colon
&
 \ar@(ul,dl)@<-5.5ex>[]_{\phi_\bullet}
 &
  A_{PL}(\Delta^{_\bullet}) \ar@<0.75ex>[r]^-{p_\bullet}
  &
   {C^*(\Delta^\bullet),} \ar@<0.75ex>[l]^-{i_\bullet} }
$$
where $p_{\bullet}$, $i_{\bullet}$ and $\phi_{\bullet}$ preserve the simplicial structure.
\end{theoremb}
The transfer diagram $\fracd_{\bullet}$ is explicitly described  for instance in \cite[\S 3]{getz}.
As usual,    $A_{PL}(  \Delta^n)=\Lambda(t_0,\dots,t_n,dt_0,\dots,dt_n)/(\sum t_i-1,\sum dt_i)$
and the maps $p_{\bullet}$ and $i_{\bullet}$ are defined as follows:

Let  $\alpha_{i_0\dots i_k}$ be the basis for $C^*(\Delta^n)$ defined by
$$\langle \alpha_{i_0\dots i_k}, a_{j_0\dots j_k}\rangle = \left\{\renewcommand{\arraystretch}{1.5}\begin{array}{ll}
(-1)^{\frac{k(k-1)}{2}} \, & \mbox{if }(j_0, \dots , j_k) = (i_0,\dots , i_k),\\
\qquad 0 & \mbox{otherwise.}
\end{array}\renewcommand{\arraystretch}{1}
\right.$$
Then,  $
 i_n(\alpha_{i_0\dots i_k}) $ is the Whitney elementary form
 $\omega_{i_0\dots i_k} $ defined by
 $$
 \omega_{i_0\dots i_k}  = k! \sum_{j=0}^k (-1)^j t_{i_j} dt_{i_0}\cdots \widehat{dt_{ij}}\cdots dt_{i_k}.$$
 Since $d(\omega_{i_0\dots i_k}) = \sum_q \omega_{qi_0\dots i_k}$, the map
 $i_n\colon C^*(\Delta^n)\to A_{PL}(\Delta^n)$ is a morphism of cochain complexes.
The map $p_{n}\colon A_{PL}(\Delta^n)\to C^*(\Delta^n)$ is defined by
$$p_n(\omega) = \sum_{k=0}^n \sum_{i_0<\dots < i_k} \alpha_{i_0\dots i_k} I_{i_0\dots i_k}(\omega),$$
with
$$I_{i_0\dots i_k}(t^{a_1}_{i_1}\dots t_{i_k}^{a_k} dt_{i_1}\dots dt_{i_k})= \frac{a_1!\cdots a_k!}{(a_1+\dots + a_k+k)!},$$
 and $0$ otherwise.
 In particular, $I_{i_0\dots i_k}( \omega_{i_0\dots i_k} )=1$.

Since the morphisms $p_\bullet$ and $i_\bullet$ preserve the simplicial structures they preserve sub simplicial complexes. So, if $K$ is a finite simplicial complex then $K\subset \Delta^n$ and the diagram ${\fracd}$ is obtained by restriction.

\begin{corollary}
For  a finite simplicial complex $K$, we have a quasi-isomorphism of differential graded Lie coalgebras,
$$I\colon (\Lc(sC^*(K)),d)\stackrel{\simeq}{\longrightarrow} (\Lc(sA_{PL}(K)),d).$$
\end{corollary}

\vspace{3mm} Let $a_0$ be a base point. The associated augmentation of $A_{PL}(K)$ is defined by   $$
\varepsilon (t_0) = 1, \quad \varepsilon (t_i) = 0 \qquad\mbox{for } i\neq 0.$$
An augmentation in the cochain complex $C^*(K)$ is also defined by $\varepsilon (a_0)=1$ and $\varepsilon (a_i)= 0$ for $i>0$.
In the transfer diagram the morphisms $p_n$ and $i_n$ preserve the augmentations because
$p_n(t_0) = a_0$ and $i_n(a_0) = t_0$. Therefore, the morphism $I$ induces a quasi-isomorphism
$$(\Lc(s\overline{C}^*(K)),d)\stackrel{\simeq}{\longrightarrow} {\mathcal E}(A_{PL}(K)).$$

\begin{theoremb}\label{previo}
With the above notations,
\begin{enumerate}
\item[(i)] The dual of the simplicial differential graded Lie coalgebra
$\bigl(\Lc(s{C}^*(\Delta^{\bullet})),D\bigr)^{\sharp}$
is a sequence of compatible models of $\mathbf\Delta$.
\item[(ii)] The differential graded Lie algebras $(\widehat{\mathbb L}^c\bigl(sC^*(K)),d\bigr)^\sharp$ and  $\lasu (K)$ are isomorphic.
\end{enumerate}
\end{theoremb}

\begin{proof}
(i) Note first that, by construction, the linear part of the differential is the good one. Thus by  \thmref{unicidadcompleta},  we   need to check both that, for all $n\ge 0$,
each  vertex in $(\Lc(s{C}^*(\Delta^{\bullet})),D)^{\sharp}$ is a Maurer-Cartan element, and  that the cofaces are as in  (\ref{coface}).

On the one hand,
By naturality it is enough to check that the only vertex $a_0$ in $(\Lc(s{C}^*(\Delta^{0})),D)^{\sharp}$ is a Maurer-Cartan element.
Note that ${A}_{PL}(\Delta^0)= \mathbb Q \cdot e= {C}^*(\Delta^0)$, with $e^2=e$ the unit of the algebra.
Now in the bar construction \cite[Page 269]{FHT}, $d([se\vert se]) = -[se]$.
The anti-symmetrisation of the diagonal gives $\Delta_L([se\vert se]) = 2[se, se]$.
In the dual Lie algebra this gives $d(a_0)= -\frac{1}{2}[a_0,a_0]$.

On the other hand, a simple inspection shows that the faces on $C^*(\Delta^\bullet)$ become the usual cofaces when taking the dual of the free lie coalgebra.

(ii) This is a direct consequence of the fact that $K\subset \Delta^n$ for some $n$ and $\lasu(K)$ is isomorphic to a sub DGL $(\widehat\lib(V),\partial)\subset \lasu_n$ where $(V,\partial_1)$ is the suspension of the chain complex of $K$.
\end{proof}

\vspace{3mm} Next, we connect the functor $\lasu$ with the classical Quillen functor $\lambda\colon \catss_{1}\to \catdgl$   from the category of
$2$-reduced  simplicial sets \cite{qui} and with Sullivan minimal models \cite{Su}.

\begin{theoremb}\label{thm:quillenetlasu}
 Let $K$ be a  connected finite simplicial complex with Sullivan minimal model $(\land V,d)$.
Then:
\begin{enumerate}
\item[(i)] For any vertex $a$, $(\lasu(K),\partial_{a})$ is quasi-isomorphic to $\left({\mathcal E}(\land V,d)\right)^{\sharp}$.
\item[(ii)] In particular, if $K$ is 1-connected, then $(\lasu(K),\partial_{a})$ is quasi-isomorphic to $\lambda (K)$.
\end{enumerate}
 \end{theoremb}

\begin{proof}   (i)
We have a sequence of quasi-isomorphisms
$$\xymatrix@1{
\cE(\land V,d)\ar[r]^-{\simeq}&
\cE(A_{PL}(K))&
(\Lc(s\ov{C}^*(K)),d).\ar[l]_-{\simeq}
}$$
By taking duals we get a series of DGL quasi-isomorphisms
$$\xymatrix@1{
\cE(\land V,d)^\sharp &
\cE(A_{PL}(K))^\sharp \ar[l]_-{\simeq} \ar[r]^-{\simeq} &
(\Lc(s\ov{C}^*(K)),d)^\sharp.
}$$
Now remark that $(\Lc(s\ov{C}^*(K)),d)^\sharp= (\lasu (K)/(a_0), \overline{\partial})$ and recall from Proposition \ref{prop:minimodel} that $(\lasu (K)/(a_0), \overline{\partial})\cong (\lasu (K), \partial_{a_0})$.

(ii) When $K$ is 1-connected, the minimal model $(\land V,d)$ of $K$  is of finite type. Thus,  it is the dual of a differential graded coalgebra $(\land V,d)^\sharp$, and
$$\xymatrix@1{
\cL((\land V,d)^{\sharp})=(\cE(\land V,d))^{\sharp}.
}$$
The result is then  a direct
consequence of a theorem of Majewski \cite{ma} who connects $\lambda(K)$ with $\cL((\land V,d)^{\sharp})$
by a sequence of quasi-isomorphisms.
\end{proof}

%%%%%%%%%%%%%%%
\section{\bf Representability of the Quillen realization functor}\label{sec:representable}

We have proved that $\lasu(K)$ is homotopy equivalent to the Quillen construction
$\lambda(K)$, when $K$ is $2$-reduced of finite type.
We study now the connection between the realization functor $\langle-\rangle$ introduced in \secref{sec:larealizacion} with the composite of the  functors,
 $$
 \mathcal C^*= (-)^\sharp \circ {\mathcal C}\colon\catdgl_{f} \to\catcdga
 \quad
 \text{and}
 \quad\langle-\rangle_S\colon \catcdga\to\catss,
 $$
 where $\catdgl_{f}$ is the full subcategory of $\catdgl$ of finite type DGL.
 The first functor is the classical cochain functor defined on the category of  Lie algebras of finite type  and
 the second one is the Sullivan's realization functor defined by
 $\langle A\rangle_S= \catcdga\bigl(A,A_{PL}(\Delta^{\bullet})\bigr)$.

\begin{theoremb}\label{lestrella}
Let $(L,d)$ be a finite type DGL with $H_q(L,d)= 0$ for $q<0$. Then,  there is a homotopy equivalence of simplicial sets,
 $$\langle L\rangle \simeq \langle \mathcal C^*(L)\rangle_S.$$
\end{theoremb}

\begin{proof}
Given the simplicial  differential graded Lie coalgebra $\lasu^c_\bullet = \bigl(\Lc(s{C}^*(\Delta^\bullet)),D\bigr)$
of \propref{previo}
 and $L\in \catdgl_{f}$, we define   the simplicial set
$$\langle L\rangle'_\bullet := \mbox{\bf DGLC}(L^{\sharp}, \lasu^c_\bullet).$$

A coalgebra morphism   $L^\sharp \to \lasu_n^c$ is completely determined by its projection on the indecomposables $L^\sharp \to sC^*(\Delta^n)$ and an algebra morphism $\lasu_n\to L$ by its restriction $s^{-1}C_*(\Delta^n) \to L$. Since $L$ is finite type, $\langle L\rangle'_n = \langle L\rangle_n$.
Also, as taking dual is compatible with the simplicial structure, the dual process induces a bijection of simplicial sets $(-)^\sharp
 {\colon} \langle L\rangle '_\bullet\to \langle L\rangle_\bullet$.

The quasi-isomorphism $I  \colon \lasu^c_\bullet\to {\mathcal E}(A_{PL}(\Delta^\bullet))$
of \propref{previo} induces a morphism of simplicial sets
$$\Psi\colon \mbox{\bf DGLC}(L^{\sharp}, \lasu^c_\bullet) \to \mbox{\bf DGLC} (L^{\sharp}, {\mathcal E}A_{PL}(\Delta^\bullet)).$$
On the other hand the functor $\cA$ is equipped with a natural evaluation morphism,
$$\varepsilon_A\colon {\mathcal A}{\mathcal E }(A,d)= (\land s\mathbb L^c(s\overline{A}), \partial) \to (A,d),$$
which vanishes  on $(\mathbb L^c)^{\geq 2}(\overline{A})$.
We then obtain a sequence of simplicial maps,
$$\xymatrix{
\mbox{\bf DGLC}(L^{\sharp},\lasu^c_{\bullet})\ar[r]^-{\Psi}&
\mbox{\bf DGLC}(L^{\sharp}, {\mathcal E}A_{PL}(\Delta^\bullet))\ar[d]^{\mathcal A} &
\\
&
\mbox{\bf CDGA} ({\mathcal A} (L^{\sharp}), {\mathcal A}{\mathcal E}A_{PL}(\Delta^\bullet))\ar[d]^{\varepsilon_*}&
\\
&
\mbox{\bf CDGA}({\mathcal C}^*(L), A_{PL}(\Delta^\bullet))=  \langle {\mathcal C}^* (L)\rangle_S.&
}$$
Recall now that in the transfer, $i_n(\alpha_{0\dots n})= \omega_{0\dots n} = n!\,dt_1\cdots dt_n$. Therefore, the linear part of $\lasu_n^c\to {\mathcal E}A_{PL}(\Delta^n)$ is $i_n$.
As $\pi_n\langle {\mathcal C}^*(L)\rangle$ can be identified with the linear maps $\mbox{Hom}(H_{n-1}(L),\mathbb Q)\to \mathbb Q \,\alpha_{0\dots n}$ (see \cite[\S 1.7]{FHTII}), this composition induces an isomorphism on the homotopy groups.
\end{proof}

%\begin{corollary}\label{ctruc} Let $(L,d)$ be a finite dimensional DGL with $H_q(L,d)= 0$ for $q<0$. Then $\langle L\rangle$ is a rational nilpotent %space and there are group isomorphisms
%$$\pi_n\langle L\rangle \cong H_{n-1}(L,d),$$
%where for $n=1$ the group structure on the right is given by the Baker-Campbell-Hausdorff product.
%\end{corollary}

%\begin{proof} By Theorem \ref{lestrella}, we have $\pi_n\langle L\rangle = \pi_n \langle \mathcal C^*(L)\rangle_S$. The result now follows from %\cite[Theorem 2.4]{FHTII}.\end{proof}

Recall now that a group $G$ is  \emph{$\mathbb Q$-complete} if each $G^n/G^{n+1}$ is a $\mathbb Q$-vector space and $G= \varprojlim_n G/G^n$. Here $G= G^1\supset G^2\supset\cdots$ is the lower central series of $G$. That is, $G^2$ is the subgroup generated by the commutators $[a,b]$ and $G^n$ is the subgroup generated by the iterated commutators $[g_1,[g_2,[\cdots [g_{n-1},g_n]\cdots ]$. A nilpotent space $X$ is called a rational nilpotent space if its fundamental group is $\mathbb Q$-complete, if each $\pi_n(X)$ is a finite dimensional $\mathbb Q$-vector space for $n\geq 2$, and if $\pi_1(X)$ acts nilpotently on  $\pi_n(X)$ for $n\geq 2$.

\begin{proposition}  Let $K$ be a finite connected simplicial complex. Then, the simplicial set $\langle \lasu (K),\partial_a\rangle$ is the projective limit of a tower of rational nilpotent spaces.
\end{proposition}

\begin{proof} Write $(\lasu (K), \partial_a) = (\widehat{\mathbb L}(V), \partial)$. Then
$$(\lasu (K), \partial_a) = \varprojlim_n (\widehat{\mathbb L}(V)/\mathbb L^{>n}, \overline{\partial}).$$
Since the realization functor is a right adjoint, it commutes with limits and
$$\langle \lasu (K), \partial_a\rangle = \varprojlim_n \langle\,\, (\widehat{\mathbb L}(V)/\mathbb L^{>n}, \overline{\partial})\,\,\rangle.$$
Now, since $(\widehat{\mathbb L}(V)/\mathbb L^{>n}, \overline{\partial})$ is a finite dimensional nilpotent Lie algebra,  its geometric realization is a rational nilpotent space in view, for instance, of Theorem \ref{lestrella}.
\end{proof}

\section{The Malcev completion of the fundamental group}
 This section is entirely devoted to the proof of the following result.
\begin{theoremb}\label{cor:QMalcev}
Let $K$ be a connected, finite simplicial complex. Then,
$H_0(\lasu(K), \partial_a)$ is the Malcev Lie completion of the fundamental group $\pi_1(K)$.
\end{theoremb}

\begin{proof} Let $(\land V,d)$ be the minimal Sullivan model of $K$.  Since $K$ is connected,  $V = V^{> 0}$. The graded vector space $L_k(\land V,d) = (V^{k+1})^{\sharp}$ equipped with the quadratic part of the differential $d$ inherits the structure of a graded Lie algebra \cite[\S2]{FHTII}, whose component $L_0(\land V,d)$ is the Lie algebra associated to the Malcev completion of $\pi_1(K)$ \cite[Theorem 7.5]{FHTII}.

Suppose first that $V$ is a finite type vector space. In that case $(\land V,d)$ is the dual of a commutative differential graded coalgebra $(C,d)$. Thus, by \cite{nei}, $H_0(  {\mathcal L}(C,d)) \cong L_0(\land V,d)$. On the other hand, the natural map ${\mathcal L}(C,d) \to \left({\mathcal E}(\land V,d)\right)^\sharp$ is the inclusion
 $$\varphi {\colon} (\mathbb L(s^{-1}\overline{C}),d)\hookrightarrow (\widehat{\mathbb L}(s^{-1}\overline{C}),d).$$
 Finally, since $ (s^{-1}V)^\sharp$ injects in the homology of the two DGL's, and is in an isomorphism in both cases, the injection $\varphi$ is a quasi-isomorphism.

 In the general case, $(\land V,d)$ is the increasing union of finite type CDGA's of the form $(\land V(n),d)$, i.e.,
 $$(\land V,d) = \varinjlim_n (\land V(n),d),$$
 where $$L_0(\land V(n),d) = L_0(\land V,d)/ L_0^n(\land V,d),$$
 $$L_0^2(\land V,d) = [L_0(\land V,d), L_0(\land V,d)], \quad L_0^n(\land V,d) = [L_0(\land V,d), L_0^{n-1}(\land V,d)],\quad n> 2.$$

 Moreover,
 $$L_0(\land V,d) = \varprojlim_n L_0(\land V,d)/ L_0^n(\land V,d).$$

We write
$$\bigl(\widehat{\mathbb L}(W(n)), \partial\bigr) = {\mathcal E}(\land V(n),d)^\sharp.$$ It follows that ${\mathcal E}(\land V,d)= \varinjlim_n {\mathcal E}(\land V(n),d)$
 and
 $$(\lasu_K, \partial_a)  {\cong} \varprojlim_n \left( \, {\mathcal E}(\land V(n),d)\right)^{\sharp} = \varprojlim_n \,(\widehat{\mathbb L}(W(n)), \partial).$$

By construction $W(n) = W(n)_{\geq 0}$ is a finite type graded vector space. We write $Y_n = \widehat{\mathbb L}(W(n))_0$, $X_n = \widehat{\mathbb L}(W(n))_1$ and $X_n^s = {\mathbb L}^s (W(n))_1$. We denote by $\rho_n {\colon} X_n\to X_{n-1}$ and $\sigma_n {\colon} Y_n\to Y_{n-1}$ the surjective morphisms induced by the inclusions $\land V(n-1)\hookrightarrow \land V(n)$. For $r<q$ we also denote by $\rho_{qr}{\colon} X_q\to X_r$ the composition $\rho_{r+1}\circ \cdots \circ \rho_q$.

We consider then the   short exact sequences of towers of vector spaces
$$\xymatrix{\mbox{} & 0\ar[d] & 0\ar[d] \\ \cdots\mbox{}\ar[r] & \partial(X_n) \ar[r]\ar[d] & \partial(X_{n-1}) \ar[r]\ar[d] & \mbox{}\cdots\\
\cdots\mbox{}\ar[r] & Y_n \ar[r]\ar[d] & Y_{n-1} \ar[r]\ar[d] & \mbox{}\cdots\\
\cdots\mbox{}\ar[r] & H_0\bigl(\widehat{\mathbb L}(W(n)), \partial\bigr) \ar[r]\ar[d] & H_0\bigl(\widehat{\mathbb L}(W(n-1)), \partial\bigr) \ar[r]\ar[d] & \mbox{}\cdots\\
\mbox{} & 0 & 0}
$$
This induces an exact sequence
$$0 \to \varprojlim_n \partial (X_n) \to \varprojlim_n Y_n \to \varprojlim_n H_0\bigl(\widehat{\mathbb L}(W(n)),\partial\bigr) \to \varprojlim^{\mbox{}\hspace{6mm}1}_n \partial(X_n).$$
Since every $\rho_n {\colon} X_n\to X_{n-1}$ is surjective, the induced map $\partial(X_n)\to \partial(X_{n-1})$ is also surjective and so $\varprojlim^1 \partial(X_n) = 0$.

Now, by the next Lemma, $\varprojlim \partial(X_n) = \partial(\varprojlim X_n)$. Therefore,
$$H_0(\lasu (K), \partial_a) \cong \varprojlim_n H_0\bigl(\widehat{\mathbb L}(W(n)), \partial\bigr) = \varprojlim_n L_0(\land V,d)/L_0^n(\land V,d)$$
is the Malcev Lie completion of $\pi_1(K)$.
\end{proof}

\begin{lemma} With the notation in the previous proof, let $a= \varprojlim (a_n)\in Y$. If for each $n$, $a_n=\partial c_n$ for some $c_n$,    then there is an element $b = \varprojlim (b_n)\in X$ such that $\partial b_n=a_n$ for all $n$.
\end{lemma}

\begin{proof} The element $a_1$ can be written as $a_1= \sum_{q=p}^\infty a_1^q$ with $a_1^q\in {\mathbb L}^q(W(1))_0$ and $a_1^p\neq 0$.  For fixed $n$ we denote by $f(n)$ the maximum integer such that there is an element $e \in \widehat{\mathbb L}^{\geq f(n)}(W(n))_1$ with $\partial e = a_n$. Since $d\circ \rho_{n1}(e ) = a_1$, we have $f(n) \leq p$. As we can replace $e_1 ,\dots , e_{n-1}$ by the images of $e_n$, the function $f(n)$ is decreasing and we denote by $r$ its limit.

We construct a sequence $\beta^r = (\beta_p^r)$ with the following properties:
\begin{enumerate}
\item[(i)] $\beta_p^r \in X_p^r$ and $ \rho_p(\beta_p^r) = \beta_{p-1}^r$.
\item[(ii)] There exist elements $e_p\in X_p$ with $\partial e_p= a_p$ and $e_p-\beta_p^r \in X_p^{>r}$.
\end{enumerate}
Suppose we have constructed $\beta_1^r, \dots , \beta_{n-1}^r$ with the two preceding properties for $p<n$ and
\begin{enumerate}
\item[(iii)] for all $q$ there exists an element $e_q\in X_q^{\geq r}$ with $\partial e_q = a_q$ and $\rho_{q(n-1)}(e_q)-\beta_{n-1}^r\in X_{n-1}^{>r}$.
\end{enumerate}
Then, we denote by $q_{n-1}^r {\colon} X_{n-1}\to X_{n-1}^r$ the projection on the component of length degree $r$ and we write
$$W_q= q^r_n\rho_{qn}\, \left( \left[ \partial^{-1}(a_q)\cap X_q^{\geq r}\right] \cap \left( q^r_{n-1}\circ \rho_{q(n-1)}\right)^{-1} (\beta_{n-1}^r)\right).$$
The $W_q$ form a decreasing sequence of affine subspaces of $X_n^r$. Since $X_n^r$ is finite dimensional and the $W_q$   are non empty, their intersection is not empty and we choose an element $\beta_n^r$ in this intersection. This ends the construction of $\beta^r= (\beta_n^r)$.

Now consider $\alpha' = \alpha - \partial\beta^r$. We proceed in the same way for $\alpha'$ and note that by construction there is an element $e'_n=   e_n- \beta_n^r \in {\mathbb L}^{\geq r+1}$ with $\partial e_n' = a_n'$. This gives an element $\beta^{r+1}$, and by induction elements $\beta^q$ for $q\geq r$.
Now
$\sum_q \beta^q = (\sum_q\beta_n^q)$ is a well defined element in $\varprojlim_n \widehat{\mathbb L}(W(n))$ and by construction $\partial (\sum_q\beta^q) = a$.
\end{proof}

%%%%%%%%%%%%%%%%%%%%%%%%%%%%%%
\providecommand{\bysame}{\leavevmode\hbox to3em{\hrulefill}\thinspace}

{\small
\vspace{5mm}\noindent {\sc Departamento de Algebra, Geometr\'ia y Topolog\'ia,
Universidad de M\'alaga,
Ap. 59,
29080-M\'alaga,
Espa\~na}

{\it E-mail address:} ubuijs@uma.es
\\[2mm]
{\sc  Institut de Math\'ematiques et Physique,
Universit\'e Catholique de Louvain-la-Neuve,
Louvain-la-Neuve,
Belgique}

{\it E-mail address:}  Yves.felix@uclouvain.be
\\[2mm]
{\sc Departamento de Algebra, Geometr\'ia y Topolog\'ia,
Universidad de M\'alaga,
Ap. 59,
29080-M\'alaga,
Espa\~na}

{\it E-mail address:} aniceto@uma.es
\\[2mm]{\sc D\'epartement de Mathematiques,
         UMR-CNRS 8524,
         Universite de Lille~1,
         59655 Villeneuve d'Ascq Cedex,
         France }

         {\it E-mail address:} Daniel.Tanre@univ-lille1.fr}
%%%%%%%%%%%%%%%%%%
\end{document}